\documentclass[amsfonts]{amsart}
\usepackage{amssymb, amsfonts , tikz-cd, bm}
\usepackage[all,arc]{xy}
\usepackage{enumitem}
\usepackage{mathrsfs}
\usepackage{microtype}
\usepackage{amsaddr}
\usepackage{amsthm,etoolbox}
\usepackage{eucal, mathbbol}

\usepackage{hyperref}
\hypersetup{hidelinks, colorlinks}
\hypersetup{
    citecolor = {teal},
    linkcolor = {black},
}
\usepackage[capitalise]{cleveref}

\theoremstyle{plain}
\newtheorem{thm}{Theorem}[section]
\newtheorem{cor}[thm]{Corollary}
\newtheorem{prop}[thm]{Proposition}
\newtheorem{lem}[thm]{Lemma}
\newtheorem{conj}[thm]{Conjecture}

\theoremstyle{definition}
\newtheorem{defn}[thm]{Definition}

\newtheorem{assm}[thm]{Assumption}

\theoremstyle{remark}
\newtheorem{rem}[thm]{Remark}
\AtBeginEnvironment{rem}{%
    \small
  \pushQED{\qed}%
}
\AtEndEnvironment{rem}{\popQED}

\numberwithin{equation}{section}

\calclayout

\AtBeginDocument{%
   \def\MR#1{}
}

\newcommand{\sg}{{\mathbb{Z}/p}}

\newcommand{\zpeq}{{[\![t, \theta]\!]}}

\newcommand{\brK}{(\!(z)\!)}
\newcommand{\brO}{[\![z]\!]}
\newcommand{\gitq}{/\!\!/}

\title[mod $p$ Quantum Hikita]{3D Mirror Symmetry in Positive Characteristic}

\author{Shaoyun Bai and  Jae Hee Lee}
\email{shaoyunb@mit.edu}
\email{jaeheelee@stanford.edu}

\thanks{The first author is supported by NSF DMS-2404843.}

\begin{document}

\begin{abstract}
    Via the formulation of (quantum) Hikita conjecture with coefficients in a characteristic $p$ field, we explain an arithmetic aspect of the theory of 3D mirror symmetry. Namely, we propose that the action of Steenrod-type operations and Frobenius-constant quantizations intertwine under the (quantum) Hikita isomorphism for 3D mirror pairs, and verify this for the Springer resolutions and hypertoric varieties.
\end{abstract}

\maketitle
\section{Introduction}\label{sec:intro}

\subsection{Context and overview}\label{ssec:intro-context}
$3D$ mirror symmetry \cite{3d-mirror}, or symplectic duality \cite{BPW1, BLPW2}, dictates that symplectic resolutions should come in pairs and various layers of information should be exchanged for a dual pair of symplectic resolutions. Such duality phenomena can be approached through different lenses, and there is a large body of literature covering different aspects, including quantization, enumerative geometry, and categorification: see \cite{kamnitzer-survey} for an introduction.

The goal of this paper is to initiate the study of $3D$ mirror symmetry in characteristic $p$ where $p$ is a prime number. Here, mod $p$ coefficients appear in an asymmetric fashion under symplectic duality. Namely, on one side, we study enumerative geometry of \emph{complex} symplectic resolutions with mod $p$ coefficients, in the sense that the counts in the enumeration problems are reduced mod $p$; on the other side, we study symplectic resolutions \emph{defined over a field of characteristic $p$}.

Our formulation of such an ``arithmetic'' $3D$ mirror symmetry conjecture is based on the \emph{quantum Hikita conjecture} \cite{KMP21}, which posits that two $D$-modules associated to conical symplectic resolutions are isomorphic for dual pairs. On one hand, we have the \emph{quantum $D$-module} coming from enumerative geometry of genus $0$ curves, and on the other hand, we have the \emph{$D$-module of graded/twisted traces} arising from quantizations of symplectic singularities.

Our conjecture (\cref{conj:intro})  based on the quantum Hikita conjecture proceeds by introducing distinguished central endomorphisms of the $D$-modules that only arise in characteristic $p$, and identifying these under the isomorphism of $D$-modules. On the quantum $D$-module side, these endomorphisms come from a certain deformation of the classical Steenrod operations in $\mathbb{F}_p$-coefficient singular cohomology; on the $D$-module of traces side, they arise from a ``large center'' for quantizations in characteristic $p$ given by a certain deformation of the Frobenius map. We verify our conjecture (\cref{thm:example}) for the classical examples of Springer resolutions and hypertoric varieties.

\subsection{Motivations and main conjecture}
Our formulation of the conjecture is based on the notable trilogy of the \emph{Hikita conjecture} \cite{hikita}, the \emph{Hikita--Nakajima conjecture} \cite{KTWWY}, and the \emph{quantum Hikita conjecture} \cite{KMP21}. In characteristic zero, these conjectures predict that given a dual pair of conical symplectic resolutions $(X, X^!)$ over $\mathbb{C}$ with a Hamiltonian torus $T$ and $T^!$ action respectively, the following should be true.
\begin{enumerate}
    \item Let $Y$ be the affinization of $X$, then there exists an isomorphism of graded algebras
    \begin{equation*}
        \mathcal{O}(Y^T) \cong H^*(X^!; \mathbb{C}),
    \end{equation*}
    where $\mathcal{O}(Y^T)$ is the ring of regular functions of the $T$-fixed point scheme with the grading induced from the conical $\mathbb{G}_m$-action.
    \item Let $\mathcal{A}$ be the universal quantized coordinate ring of $X$. Then there exists an isomorphism of graded algebras
    \begin{equation*}
        B(\mathcal{A}) \cong H^*_{T^! \times \mathbb{G}_m} (X^!; \mathbb{C}),
    \end{equation*}
    where $B(\mathcal{A})$ is the $B$-algebra (\cref{defn:B-algebra}) of $\mathcal{A}$ with respect to a cocharacter in $X_{\bullet}(T)$, and the two factors of $T^! \times \mathbb{G}_m$ act on $X^{!}$ by Hamiltonian automorphisms of  $X^!$ and conical actions respectively.
    \item Finally, the $D$-module of graded/twisted traces of $\mathcal{A}$ (\cref{defn:Dmod-eq}) should be identified with the (specialized) quantum $D$-module of $X^{!}$ (\cref{defn:Dmod-quantum}):
    \begin{equation*}
        \mathscr{B} (\mathcal{A}) \cong QH^*_{T^! \times \mathbb{G}_m}(X^!;\mathbb{C}).
    \end{equation*}
\end{enumerate}
We refer the reader to the exposition in Section \ref{sec:quantum-hikita} for a precise discussion of the relevant concepts. Beginning from the original work of \cite{hikita} verifying the original Hikita conjecture for Hilbert scheme of points in a plane, parabolic type $A$ Slodowy slices, and hypertoric varieties, extensive efforts (though likely to be an incomplete list) \cite{nakajima-takayama, hatano2021cohomologyframedmodulispaces, shlykov2022hikitaconjectureminimalnilpotent, leung2022elliptichypertoricvarieties, BFN-Springer, krylov2023hikitanakajimaconjecturegiesekervariety, chen2024quantizationminimalnilpotentorbits, hoang2024hikitaconjectureclassicallie, setiabrata2024hikitasurjectivitymathcaln, hoangkrylovmatvieievskyi} have been dedicated to prove these conjectural isomorphisms for concrete examples. We refer to \cite[Section 1.3]{hoangkrylovmatvieievskyi} for an excellent survey of the current status of the Hikita-type conjectures for various examples.

Our motivation for considering the mod $p$ quantum Hikita conjecture comes from the manifestation of the ``large center" phenomena in $D$-modules with mod $p$ coefficients in both quantization theory and enumerative geometry. Namely, the respective $D$-modules in the statement of quantum Hikita conjecture carry extra central endomorphisms in mod $p$ coefficients with different origins, from deformation quantization in characteristic $p$ and mod $p$ curve counts. Our main proposal is to identify these central endomorphisms of $D$-modules. 

For the statement, let $k$ be a field of characteristic $p$. Let $X, X^!$ be a pair of symplectically dual conical symplectic resolutions equipped with Hamiltonian torus actions of $T, T^!$ resp., but now assume that $X$ and the torus $T$ are defined over the ground field $k$ (admits a ``split form''), while $X^!$ and $T^!$ are still both defined over $\mathbb{C}$. For a more precise formulation of what we mean by symplectically dual pairs in this setting, see the discussion in Section \ref{ssec:hikita-quantum-p}.

For quantizations of algebraic symplectic varieties over a field of positive characteristic, Bezrukavnikov--Kaledin \cite{bezrukavnikov-kaledin-quantp} introduced \emph{Frobenius-constant quantizations}, which, roughly speaking, are quantizations of the coordinate rings of symplectic varieties carrying a central lift of the Frobenius endomorphism. Such quantizations exist for a large class of examples, including Springer resolutions \cite{bezrukavnikov-mirkovic-rumynin} and Coulomb branches \cite{BFN2} for $3d$ $\mathcal{N}=4$ gauge theories \cite{lonergan}. A simple calculation shows that the central elements coming from such a lift of the Frobenius induce central endomorphisms of $D$-modules of graded/twisted traces:

\begin{prop}[\cref{prop:frob-acts-on-Dmod}]
    Let $\mathcal{X}$ be the universal deformation (\cref{ssec:dymp-res}) and $\mathcal{A}$ be the global sections of its quantization. Assume that there is a \emph{Frobenius-constant quantization} (\cref{defn:frob-const-quant}) $\Lambda: \mathcal{O}(\mathcal{X})^{(1)} \to Z(\mathcal{A})$. Then $\Lambda(a)$ for $a \in \mathcal{O}(\mathcal{X})_0$, an element of weight zero under $T$ (see \cref{ssec:dymp-res}), acts on the $D$-module of twisted traces as a central endomorphism.
\end{prop}

On the quantum $D$-module side, \emph{quantum Steenrod operations} \cite{Fuk97, Wil20} and their equivariant generalizations \cite{Lee23a, Lee23b} are deformations of the classical Steenrod operations on cohomology with $\mathbb{F}_p$-coefficients, in the same way the quantum cohomology is a deformation of the cohomology ring with the cup product. They arise from suitable $\mathbb{Z}/p$-equivariant counts of genus $0$ stable maps. It is proven in \cite{seidel-wilkins, Lee23a} that the quantum Steenrod operations are covariantly constant with respect to the quantum connection and commute with each other, thereby giving rise to central endomorphisms of quantum $D$-modules (\cref{cor:qst-acts-on-Dmod}):

\begin{prop}[\cref{cor:qst-acts-on-Dmod}]
    Fix a $T^!$-equivariant cohomology class $x \in H^*_{T^!}(X^!;k)$. Then the quantum Steenrod operator $\Sigma^{T^!}_{x}$ deforming the Steenrod operations of $x$ acts on the $T^!$-equivariant quantum $D$-module as a central endomorphism.
\end{prop}

The following is a description of our conjecture. A more precise statement, which we dub the \emph{mod p quantum Hikita conjecture}, is given in \cref{conj:hikita-quantum-p}. 
 
\begin{conj}[Slogan, see \cref{conj:hikita-quantum-p}]\label{conj:intro}
    The $D$-module of twisted traces $M_{eq}$ of $X$ is isomorphic to the equivariant quantum $D$-module $M^!_{Kah}$ of $X^!$ with $k$-coefficients:
        \begin{equation*}
            M_{eq} := \mathscr{B} (\mathcal{A})_k \cong QH^*_{T^! \times \mathbb{G}_m}(X^!;k) = : M^!_{Kah}.
        \end{equation*}
    Under this isomorphism, the central endomorphisms of $M_{eq}$ induced from the Frobenius-constant quantizations of $X$ are identified with the central endomorphisms of $M_{Kah}^!$ given by quantum Steenrod operators.
\end{conj}

\begin{rem}
    One can take the ``classical limit'' of the \cref{conj:intro} by specializing the deformation parameters involved in the definition of $M_{eq}$ and $M^!_{Kah}$ to zero, thereby obtaining a mod $p$ Hikita--Nakajima conjecture. In this case, our conjecture specializes to the claim that the isomorphism in the Hikita--Nakajima conjecture intertwines the action of $B(\mathcal{O}(\mathcal{X})) \cong H^*_{T^!}(X^!)$ given by the Frobenius-constant quantizations on the trace side and the \emph{classical} Steenrod operations on the quantum side.
\end{rem}

\begin{rem}
    We emphasize the asymmetric roles played by the mod $p$ coefficients on two sides---on one hand as the coefficients of topological invariants, and on the other hand as the ground ring. Such asymmetry is a salient feature of duality statements. The same philosophy appears, for example, in arithmetic homological mirror symmetry \cite{lekili-polishchuk, seidel-formal} and the geometric Satake correspondence \cite{mirkovic-vilonen}, where the coefficients of the cohomology on one side correspond to the ground ring of definition on the other side.
\end{rem}

\subsection{Methods}
To provide evidence to our conjecture, we verify its validity in the following examples.

\begin{thm}\label{thm:example}
    Conjecture \ref{conj:intro} for $p \gg 0$ holds for dual Springer resolutions $T^* G/B \leftrightarrow T^* G^!/ B^!$ and (Gale) dual hypertoric varieties.
\end{thm}

For $k = \mathbb{C}$, the identification of $D$-modules is \cite[Theorem 1.2]{KMP21}, and establishing such an isomorphism with characteristic $p$ coefficients is a straightforward extension. Matching the central endomorphisms on two sides is the true content, whose proof should be interesting on its own. We briefly highlight the underlying idea as follows.

Consider the Weyl algebra over $\mathbb{F}_p$
\begin{equation*}
    \mathcal{D}_\hbar(\mathbb{A}^1) := \mathbb{F}_p[\hbar]\langle x, \partial\rangle / ([\partial, x] = \hbar).
\end{equation*}
Then due to the characteristic $p$ nature, it is straightforward to check that $x^p$ and $\partial^p$ are central elements. A more interesting central element is the so-called \emph{Artin--Schreier} expression,
\begin{equation*}
    x^p \partial^p = \prod_{i=0}^{p-1} (x \partial - i \hbar) = (x \partial)^p - \hbar^{p-1}(x\partial).
\end{equation*}
It arises naturally in the following three contexts.
\begin{enumerate}
    \item As described above, quantization of Poisson algebras in characteristic $p$ admit a ``large center'': the image of $xy \in \mathbb{F}_p[x,y] = \mathcal{O}(T^*\mathbb{A}^1)$ under the Frobenius-constant quantization $\Lambda: \mathcal{O}(T^*\mathbb{A}^1)^{(1)} \to Z(\mathcal{D}_\hbar (\mathbb{A}^1))$ is given by
    \begin{equation*}
        \Lambda(xy) = (x\partial)^p - \hbar^{p-1} (x\partial).
    \end{equation*}
    \item Let $T$ be the $1$-dimensional complex torus; then $H^*(BT; \mathbb{F}_p) = \mathbb{F}_p [\![u]\!]$ where $\mathrm{deg}(u) = 2$. Then the total Steenrod operation $\mathrm{St}_p$ satisfies
    \begin{equation*}
        \mathrm{St}_{p}(u) = u^{p} - \hbar^{p-1}u,
    \end{equation*}
    where $\hbar$ is the degree $2$ element of the group cohomology of $\mathbb{Z}/p$ with $\mathbb{F}_p$-coefficients.
    \item Consider a ($\hbar$-)connection over $\mathbb{F}_p(\!(z)\!)$ with regular singularity at $z = 0$
    \begin{equation*}
        \nabla_{z\partial_z} = \hbar z\frac{d}{dz} + A(z), \quad \quad \quad A \in \mathrm{Mat}(n, \mathbb{F}_p[\![z]\!]),
    \end{equation*}
    then the \emph{$p$-curvature} of $\nabla$, which is a $\mathbb{F}_p(\!(z)\!)$-linear endomorphism, has the form
    \begin{equation*}
        \nabla^p_{z\partial_z} - \hbar^{p-1}\nabla_{z\partial_z}.
    \end{equation*}
\end{enumerate}
Our proof of Theorem \ref{thm:example} is based on reducing the covariantly constant endomorphisms of the $D$-modules to Artin--Schreier type expressions.

For the $D$-module of twisted traces, we show that the image of Frobenius-constant quantization map $\Lambda: \mathcal{O}(\mathcal{X})^{(1)} \to Z(\mathcal{A})$ from the universal deformation of $X$ acts by left multiplication (\cref{prop:frob-acts-on-Dmod}). The map $\Lambda$ generalizes Context (1) from quantization theory to the symplectic resolutions we consider. The map $\Lambda$ can be described explicitly for the Springer resolution $T^*G/B$  \cite{bezrukavnikov-mirkovic-rumynin} and for hypertoric varieties (Coulomb branches of abelian gauge theories) \cite{lonergan}. These are both eventually described by the ring of crystalline differential operators, in which the Artin--Schreier expressions enter in a direct way.

\begin{rem}
    Our result is greatly inspired by the pioneering work of Lonergan \cite{lonergan}, in which quantized Coulomb branches with $\mathbb{F}_p$-coefficients are given a Frobenius-constant structure with Frobenius center constructed from Steenrod operations over the (Borel--Moore) homology of the space of Braverman--Finkelberg--Nakajima triples. Therefore, these central endomorphisms for Coulomb branches are eventually related to Context (2).
\end{rem}

For quantum $D$-modules, Context (3) enters through the conjecture of the second-named author \cite{Lee23b}. Here, let $X^!$ be a $T^{!}$-equivariant conical symplectic resolution over $\mathbb{C}$. In \cite{Lee23b}, generalizing the work of \cite{Fuk97, Wil20, seidel-wilkins} to the $T^!$-equivariant setting, central endomorphisms $\Sigma^{T^!}$ of the quantum $D$-module were constructed and related to the $p$-curvature of the quantum $D$-module: 

\begin{conj}\label{conj:p-curvature}
    Given a degree $2$ cohomology class $x \in H^2_{T^!}(X^! ; \mathbb{F}_p)$, the $T^!$-equivariant quantum Steenrod operator along $x$ (cf. Section \ref{ssec:oper-qst}) $\Sigma^{T^!}_x$ satisfies
    \begin{equation*}
        \Sigma^{T^!}_x = (\nabla_x^{T^!})^p - t^{p-1} \nabla_x^{T^!},
    \end{equation*}
    where $\nabla_x^{T^!} = t \partial_x + x \star_{T^!}$ is the differential operator defining the quantum $D$-module. Here, $t$ with $\deg(t) = 2$ is the \emph{loop rotation} equivariant parameter and $\star_{T^!}$ is the ${T^! \times \mathbb{G}_m}$-equivariant quantum multiplication.
\end{conj}
Conjecture \ref{conj:p-curvature} is established in \cite{Lee23b} for Springer resolutions and we verify it for hypertoric varieties in this paper. 

Then, by appealing to the isomorphisms established in the original quantum Hikita conjecture for Springer resolutions and hypertoric varieties \cite{KMP21}, we prove Theorem \ref{thm:example} by a direct comparison. Namely, we show that the action from ``degree 2'' classes can be identified with the $p$-curvature of the respective $D$-modules, and appeal to the fact that such degree $2$ classes determine the action in these examples as they generate.

\subsection{Related work and speculations}

The fact that the conical symplectic resolutions arising in 3D mirror symmetry have lifts to integer coefficients was observed as early as in Braverman--Finkelberg--Nakajima's original work in Coulomb branches of gauge theories, \cite[Remark 2.4]{BFN2}. In particular, by taking mod $p$ reductions, we naturally obtain objects of arithmetic algebraic geometry.

This aspect was leveraged crucially in the inspirational work of Lonergan \cite{lonergan}, where Frobenius-constant quantizations of mod $p$ Coulomb branches are constructed via Steenrod operations (the version for equivariant Borel--Moore homology). Lonergan's work highlights the important similarity between Frobenius-constant quantizations and Steenrod operations which serves as the starting point of our discussion: for dual pairs of symplectic resolutions given by Higgs/Coulomb branches in supersymmetric gauge theories, our conjecture can be interpreted as the claim that the (quantum) Steenrod operations in the mod $p$ cohomology of Higgs branches are dual to Lonergan's Frobenius-constant quantization.

The relationship between quantum Steenrod operations and $p$-curvature of quantum connections was noted in \cite{Lee23b}. In the context of arithmetic 2D mirror symmetry, the relationship between quantum Steenrod operations and $p$-curvature type expression in Cartan calculus in de Rham cohomology is also discussed in \cite{zihong-arxiv}. Our observation that Frobenius-constant quantizations act by $p$-curvature operators on the $D$-module of twisted traces is drawn from the construction of the Frobenius-constant quantizations themselves as $p$-curvature in crystalline differential operators in many examples \cite{bezrukavnikov-mirkovic-rumynin, bezrukavnikov-finkelberg-ginzburg, stadnik}.

We collect more speculative connections with other works as remarks below.

\begin{rem}
    Our ``quantum Hikita enhanced by power operations'' conjecture has a $K$-theoretic analogue involving $q$-difference modules, in which a $K$-theoretic version of Frobenius-constant quantizations (see \cite[Section 4]{lonergan}) and a $K$-theoretic analog of the quantum Steenrod operations (the ``quantum Adams operations'') in quantum $K$-theory play a role. The quantum Adams operations were defined by the authors in \cite{qAdams}, and they can be interpreted as enumerative geometric manifestations of the operators constructed in \cite{koroteev-smirnov}.
\end{rem}

\begin{rem}
    In view of discussion of ``BFN Springer theory'' in \cite{BFN-Springer}, there is a heuristic geometric picture behind Conjecture \ref{conj:intro} for the dual pair arising as Higgs/Coulomb branches of 3d $\mathcal{N}=4$ gauge theories. Given a reductive Lie group $G$ and a complex representation $V$, as discussed in \emph{loc. cit.}, the flat sections of quantum $D$-modules of Higgs branches associated with $(G,V)$, known as vertex functions, can be viewed as equivariant volumes of spaces showing up in the construction of the Coulomb branch associated with $(G,V)$. The similarity between the moduli spaces defining the quantum Steenrod operators and Lonergan's Frobenius-constant quantization can be viewed as a $\mathbb{Z}/p$-equivariant generalization of this philosophy. We expect that a refinement of our argument can be used to prove Conjecture \ref{conj:intro} for pairs of Higgs and Coulomb branches. A provisional long term goal of the project initiated in this paper is to understand the modular representation theory of Coulomb branches via such relationships.
\end{rem}

\begin{rem}
    In the upcoming work of Dinkins--Karpov--Krylov, the following strategy for proving the quantum Hikita conjecture for Higgs/Coulomb dual pairs for $(G,V)$ is given. First one considers the (trivial) $D$-module $H^*_{G \times T^! \times \mathbb{G}_m}(\mathrm{pt})[\![z]\!] := H^*_{G \times T^! \times \mathbb{G}_m}(\mu^{-1}(0))[\![z]\!]$ from the stack $[\mu^{-1}(0)/G] \subseteq [T^*V/G]$ where $\mu^{-1}(0)/\!\!/G$ is the Higgs branch. Then, one expresses the quantum $D$-module and the $D$-module of twisted traces as certain quotients of this $D$-module together with an identification of their solutions. The quotient map to the quantum $D$-module is given by a version of the quantum Kirwan map in quasimap quantum cohomology, and the solutions obtained as the image of these maps are the (capped) vertex functions \cite[Section 7]{Oko17}. The mod $p$ version of this statement should be related to the recent work of \cite{xu-qst} relating the classical Steenrod operations in $H^*_{G \times T^! \times \mathbb{G}_m}(\mathrm{pt})$ with quantum Steenrod operations via the quantum Kirwan map. It should be interesting to understand the relationship of this approach with the mod $p$ (truncations of) vertex functions constructed in \cite{smirnov-varchenko}.
\end{rem}

\subsection{Organization of the paper}

In section 2, we set our notations and conventions for symplectic resolutions (with a conical $\mathbb{G}_m$-action and a Hamiltonian torus action), and survey the formulation of symplectic duality via the quantum Hikita conjecture.

In section 3, we introduce the covariantly constant endomorphisms on the $D$-modules in the quantum Hikita conjecture in positive characteristic: the quantum Steenrod operators and the multiplication operators from Frobenius-constant quantizations. Then we formulate a precise version of our conjecture.

In section 4, the verification of the conjecture is provided for the case of Springer resolutions and hypertoric varieties (Higgs/Coulomb branches for abelian gauge theories). These go through the identification of the \emph{p-curvature} of the respective $D$-modules and a ``generation in degree $2$'' statement.

\subsection{Acknowledgements} The first author thanks Andrei Okounkov for illuminating discussions on Frobenius structures in enumerative geometry. The second author thanks Dan Pomerleano for introducing the inspiring paper of \cite{lonergan} to him. We thank Zihong Chen, Sanath Devalapurkar, Vasily Krylov, Calder Morton-Ferguson, Nicholas Proudfoot, Mohan Swaminathan, Ben Webster for helpful discussions at various stages of this project. In particular, the suggestion to look into quantum Hikita conjecture \cite{KMP21} was independently given to the second author by Vasily Krylov, Ben Webster, and the first author. We also thank Hunter Dinkins, Ivan Karpov, and Vasily Krylov for explaining their forthcoming work on quantum Hikita conjecture. Finally, we thank Zihong Chen, Vasya Krylov, and the anonymous referee for comments on a draft of this paper which improved the exposition greatly.

\section{Quantum Hikita conjecture}\label{sec:quantum-hikita}

Our formulation of the 3D mirror symmetry conjecture enhances the quantum Hikita conjecture proposed by \cite{KMP21}. In this section, we review the statement of the quantum Hikita conjecture by \cite{KMP21} for symplectic dual pair of symplectic resolutions; for a more detailed description, we refer the reader to the original paper. We begin by describing symplectically dual pairs, the definition of their quantum $D$-modules and the $D$-modules of twisted traces, and the quantum Hikita conjecture positing the identification of these $D$-modules for dual pairs of symplectic resolutions.

\subsection{Symplectic resolutions}\label{ssec:dymp-res}

The 3D mirror symmetry conjecture or the symplectic duality conjecture describes a duality between pairs of smooth symplectic varieties known as \emph{symplectic resolutions}. 

\subsubsection{Definition}\label{sssec:symp-res-defn}
Throughout our discussion, we work over a field $k$ which may be of positive characteristic. For a discussion of symplectic resolutions in positive characteristic and especially the theory of quantizations in characteristic $p$, we refer to the excellent \cite[Section 1-3]{kaledin-survey}.

\begin{defn}
    A \emph{symplectic resolution} is a morphism $\pi: X \to Y$ of algebraic varieties such that:
    \begin{itemize}
        \item $X$ is smooth with a symplectic structure $\omega \in H^0(X;\Omega^2_X)$;
        \item $Y$ is affine, normal, and Poisson;
        \item $\pi$ is projective, birational, and Poisson.
    \end{itemize}
    We also refer to the smooth $(X, \omega)$ as the symplectic resolution in our discussion.
\end{defn}

Given a fixed $Y$, the choice of a resolution $X \to Y$ may not be unique (they may be related by birational transformations). When we discuss symplectic resolutions, we implicitly fix a choice.

There are two main sources of such varieties: (i) the (generalized) Springer resolution associated to a pair $(G, P)$ of a reductive group $G$ and a parabolic subgroup $P$, and (ii) supersymmetric gauge theory associated to a pair $(G,V)$ of a reductive group $G$ and a complex representation $V \in \mathrm{Rep}(G)$. For (ii), the data $(G,V)$ give rise to constructions of two large classes of symplectic resolutions, (ii-1) the Higgs branches and (ii-2) the Coulomb branches. In particular, (ii) includes hypertoric varieties, resolutions of Kleinian singularities, Nakajima quiver varieties, etc., as examples. We will review these constructions briefly in the subsequent subsection. 

\begin{rem}
Note that we do not make assumptions about the characteristic over which $X \to Y$ is defined; for the smoothness of $X$ in a given choice of geometric data as above, one may need to choose a generic $p \gg 0$.
\end{rem}

\begin{defn}
    A symplectic resolution $(X,\omega)$ is \emph{Hamiltonian} if it admits a Hamiltonian action of a split algebraic torus $T$, that is $T$ acts faithfully by Poisson automorphisms on $\mathcal{O}(X)$. In particular, there is an induced $X^\bullet(T)$-grading on the coordinate ring $\mathcal{O}(X) = \bigoplus_{\lambda \in X^\bullet(T)} \mathcal{O}(X)_\lambda$, where $X^\bullet(T) := \mathrm{Hom}(T, \mathbb{G}_m)$ is the weight lattice. 
    
    We refer to this grading as the \emph{weight grading}.
\end{defn}

\begin{defn}
    A symplectic resolution $(X, \omega)$ is \emph{conical} if there exists a $\mathbb{Z}$-grading on the coordinate ring $\mathcal{O}(X)$ such that 
    \begin{itemize}
        \item The grading $\mathcal{O}(X) = \bigoplus_{n \ge 0} \mathcal{O}(X)^n$ is non-negative, with $\mathcal{O}(X)^0 = k$ and $\mathcal{O}(X)^1 = 0$.
        \item The Poisson bracket has degree $-2$.
    \end{itemize}
    We refer to this $\mathbb{Z}$-grading as the \emph{conical grading}.
\end{defn}
Note that the Poisson bracket is $(-2)$-shifted with respect to the conical grading, and hence $\mathcal{O}(X)^2$ forms a 
Lie algebra with respect to the Poisson bracket.

\begin{defn}[cf. {\cite[Section 2.1]{KMP21}}]
    A symplectic resolution $(X, \omega)$ over $k$ is \emph{conical Hamiltonian} if it is Hamiltonian, conical, and the Hamiltonian $T$-action is compatible with the conical action in the sense that
    \begin{itemize}
        \item $T$ commutes with the $\mathbb{G}_m$-action on $X$ corresponding to the conical grading,
        \item The Lie algebra $\mathfrak{t} = \mathrm{Lie}(T)$ is identified with the Cartan subalgebra of the Poisson algebra $\mathcal{O}(X)^2$ through its action on $\mathcal{O}(X)$.
    \end{itemize}
\end{defn}

For a conical Hamiltonian symplectic resolution and its $X^\bullet(T) \times \mathbb{Z}$-grading on $\mathcal{O}(X)$, we denote the $(\lambda, n)$-graded piece as $\mathcal{O}(X)^n_\lambda$.

\subsubsection{Deformations and quantizations}\label{sssec:symp-res-defm-quant}

Recall the following definition:
\begin{defn}
    A quantization $A$ of a commutative $k$-algebra $R$ is an associative, flat $k[\![\hbar]\!]$-algebra, complete with respect to the $\hbar$-adic filtration, equipped with an isomorphism $A/\hbar \cong R$.  Similarly, a quantization of a scheme $X/k$ is a sheaf $\mathscr{A}$ of associative, flat $k[\![\hbar]\!]$-algebras over $X$, complete with respect to the $\hbar$-adic filtration, equipped with an isomorphism $\mathscr{A}/\hbar \cong \mathcal{O}_X$ for the structure sheaf $\mathcal{O}_X$.
\end{defn}

Over $k = \mathbb{C}$, the general theory \cite{losev-deformation} of conical Hamiltonian symplectic resolutions $(X,\omega)$ guarantees the existence of 
\begin{itemize}
    \item Deformations of $X$: a deformation family parametrized by $H^2(X)$, fitting into a diagram
    \[
    \begin{tikzcd}
        \mathcal{X} \dar \rar & H^2(X) \dar \\
        \mathcal{Y} \rar & H^2(X)/W
    \end{tikzcd}
    \]
    with $\mathcal{X}$ as the universal Poisson deformation of $\mathcal{X}$ and $W$ is the (Namikawa) Weyl group.
    \item Universal quantizations of $Y$: quantizations of $\mathcal{X}$, whose global sections are denoted by $\mathcal{A}$. Denoting the quantization parameter by $\hbar$, this is equipped with a map $\mathcal{A} / k \hbar \to \mathcal{O}(\mathcal{X})$ to the coordinate ring of the universal deformation.
\end{itemize}

For a more detailed discussion of the deformations and quantizations of symplectic resolutions, we refer the reader to \cite[Section 4]{kamnitzer-survey}.

If $k$ is of positive characteristic, we are unaware of such a general theory. However, in many specific examples, both the universal deformation and the universal quantizations can be described explicitly, and may be verified to be split (defined over $\mathbb{Z}$ or some localization thereof). In this paper, we consider examples for which the existence of universal deformations and quantizations over some ring of integers can be verified explicitly. These provide the examples of conical Hamiltonian symplectic resolutions over fields of positive characteristic.

If $X$ is Hamiltonian, then there is an induced $X^\bullet(T)$-grading on $\mathcal{O}(\mathcal{X})$ and $\mathcal{A}$, arising from a fiberwise $T$-action for the deformation family $\mathcal{X} \to H^2(X)$. We again denote the $(\lambda, n) \in X^\bullet(T) \times \mathbb{Z}$-graded pieces as $\mathcal{O}(\mathcal{X})_\lambda^n$ and $\mathcal{A}_\lambda^n$, respectively.

Explicitly, the weight grading takes the form
\begin{equation}
    \mathcal{A} = \bigoplus_{\lambda \in X^\bullet(T)} \mathcal{A}_\lambda, \quad \mathcal{A}_\lambda := \{ a \in \mathcal{A} : [x, a] = \hbar \langle \lambda, \bar{x} \rangle a, \ x \in \mathcal{A}^2_0\};
\end{equation}
note that the commutator action of $x \in \mathcal{A}_0^2$ factors through its image $\bar{x} \in \mathfrak{t}$ under the projection of Lie algebras $\mathcal{A}_0^2 \to \mathcal{O}(X)_0^2 \cong \mathfrak{t}$.

\begin{lem}
    The classical limit map $\mathcal{A} \to \mathcal{O}(\mathcal{X})$ preserves the $X^\bullet(T)$-grading.
\end{lem}
\begin{proof}
    Consider $a \in \mathcal{A}_\lambda$. By definition, we have $[x,a] = \hbar \langle \lambda, \bar{x} \rangle a$ for all $x \in \mathcal{A}_0^2$.
    The fact that the commutator on $\mathcal{A}$ descends to $\hbar$ times the Poisson bracket implies that $\{\bar{x}, \bar{a}\} = \langle \lambda, \bar{x} \rangle \bar{a}$, in particular the $T$-weight of $\bar{a} \in \mathcal{O}(\mathcal{X})$ is indeed $\lambda \in X^\bullet(T)$.
\end{proof}

\subsection{Examples of symplectic resolutions}\label{ssec:exmp-symp-res}

\subsubsection{Springer resolutions}\label{sssec:exmp-springer}
Let $G$ be a semisimple, simply-connected split algebraic group defined over $k$, denote by $B$ a Borel subgroup and $T$ a maximal torus; the choice of $B$ gives an identification of $T$ with the abstract Cartan group $H$ of $G$ which we fix. The flag variety $G/B$ is smooth regardless of the characteristic.

In positive characteristic, we choose the characteristic of $k$ to be $p > h$, where $h$ is the maximal Coxeter number of $G$. This in particular ensures that there is a non-degenerate invariant symmetric form on $\mathfrak{g}$.

The cotangent bundle $X =  T^*(G/B)$ is a smooth variety over $k$ and admits an affinization map to the nilpotent cone $X \to \mathcal{N}$. It is called the \emph{Springer resolution} for the pair $(G,B)$. In positive characteristic, under the assumption $p > h$ on the characteristic, we may identify $\mathcal{N}$ as a subvariety of either $\mathfrak{g}$ or $\mathfrak{g}^*$, \cite[3.1.2]{bezrukavnikov-mirkovic-rumynin}.

The Springer resolution $X = T^*(G/B)$ admits the structure of a conical Hamiltonian symplectic resolution, where the action of $T$ on $X$ lifts the left multiplication action of $T$ on $G/B$. The affinization map $X \to \mathcal{N}$ gives an identification of $\mathcal{O}(X)$ with the coordinate ring on the zero fiber of the adjoint quotient map $\mathfrak{g}^* \to \mathfrak{t}^*/W$, which carries the Poincar\'e--Birkhoff--Witt grading so that $x \in \mathfrak{g} \subseteq k[\mathfrak{g}^*]$ carries degree $2$.

The universal deformation $\mathcal{X}$ of $X$ is given by the Grothendieck--Springer resolution $\widetilde{\mathfrak{g}}^*$, whose affinization agrees with $\mathfrak{g}^* \times_{\mathfrak{t}^*/W} \mathfrak{t}^*$. The quantization $\mathcal{A}= U_\hbar \mathfrak{g} \otimes_{Z(U_\hbar\mathfrak{g})} \mathrm{S} ( \mathfrak{t}) [\hbar]$ is the Rees algebra with respect to the PBW filtration, where the Rees parameter is denoted by $\hbar$.

For non-maximal parabolic $P \le G$, one still has a conical Hamiltonian symplectic resolution $T^*(G/P)$ resolving its affinization, the \emph{generalized Springer resolution}. In this paper, we exclusively consider the case $P = B$.

\subsubsection{Gauge theory: Higgs branches}\label{sssec:exmp-higgs}
Let $G$ be a complex reductive connected Lie group and $V$ be a finite-dimensional complex representation. From these data, one can construct two symplectic singularities called the \emph{Higgs branch} and the \emph{Coulomb branch}. We first briefly discuss the construction of the Higgs branch.

Take the cotangent bundle $T^* V \cong V \times V^*$ equipped with the standard symplectic form. Then the $G$-action on $V$ can be lifted to a Hamiltonian action on $T^* V$ with moment map $\Phi: V \times V^* \to \mathfrak{g}^*$. 

\begin{defn}
    The \emph{Higgs branch} is the algebraic Hamiltonian reduction
\begin{equation}
    T^* V /\!\!/\!\!/\!\!/_{0,0} G := \Phi^{-1}(0) /\!\!/_{0} G,
\end{equation}
where we take the usual GIT quotient with trivial stability condition.
\end{defn}

Deformations and resolutions of Higgs branches are both controlled by characters of $G$, as follows. After choosing a character $\chi: G \to \mathbb{G}_m$, we can take the GIT quotient
\begin{equation}
    T^* V /\!\!/\!\!/\!\!/_{0,\chi} G := \Phi^{-1}(0) /\!\!/_{\chi} G,
\end{equation}
with affinization map
\begin{equation}
        \Phi^{-1}(0) /\!\!/_{\chi} G \to \Phi^{-1}(0) /\!\!/_{0} G,
\end{equation}
which in nice cases yield a (partial) resolution of the Higgs branch.

Deformations of Higgs branches arise from allowing the moment map parameter to vary, in particular for $(\mathfrak{g}^{ab})^* := \mathrm{Hom}(\mathfrak{g}, k) \subseteq \mathfrak{g}^*$,
\begin{equation}
    \Phi^{-1}((\mathfrak{g}^{ab})^*)/\!\!/_\chi G \to \Phi^{-1}((\mathfrak{g}^{ab})^*)/\!\!/_0 G
\end{equation}
is a model for the universal deformation and its affinization.

\subsubsection{Gauge theory: Coulomb branches}\label{sssec:exmp-coulomb}

Continuing our notations introduced above, let $G$ be a complex reductive connected Lie group and suppose $V$ is a finite-dimensional complex representation. Braverman--Finkelberg--Nakajima introduced the construction of Coulomb branches associated to the data $(G,V)$. The construction may be regarded as an analog of the (generalized) Springer resolutions in the affine, parabolic setting.

\begin{defn}
    The \emph{BFN Springer resolution} is defined as $\mathcal{T}_{G, V} := G\brK \times^{G \brO} V\brO$, the quotient of $G\brK \times V\brO$ by the diagonal relation $[g(z) g_1(z), g_1(z)^{-1} v(z)] \sim [g(z), v(z)]$ for $g_1(z) \in G\brO$. The space $\mathcal{T}_{G,V}$ is equipped with two natural maps: the projection map 
    $$ \pi: \mathcal{T}_{G, V} \to \mathrm{Gr}_{G} := G\brK/G\brO, $$ 
    and the action map
    \begin{equation*}
        \begin{aligned}
            \mu : \mathcal{T}_{G, V} &\to V\brK \\
            [g(z), v(z)] &\mapsto g(z) \cdot v(z).
        \end{aligned}
    \end{equation*}
\end{defn}

\begin{rem}
    This construction is analogous to the usual Springer resolution $T^*(G/B) \cong G \times^B \mathfrak{n}$ which maps to $G/B$ (as the cotangent bundle) and $\mathcal{N}$ (by the Springer resolution, induced by the action of $G$ on $\mathfrak{n}$).
\end{rem}

There is a modular description of both $\mathrm{Gr}_G$ and $\mathcal{T}_{G,V}$. The affine Grassmannian $\mathrm{Gr}_G$ can be thought of as the space of pairs $\{ (\mathcal{P}, \phi) \}$ where $\mathcal{P}$ is a principal $G$-bundle over the formal disc and $\phi$ is a trivialization over the punctured disc, considered up to isomorphism of such pairs. The identification is given by describing the action of $G\brK$ on such space by multiplying the trivialization by an element of $G\brK$; this is a transitive action with stabilizer $G\brO$.

Similarly, the BFN Springer resolution can be described as isomorphism classes of triples $\{ (\mathcal{P}, \phi, s) \}$ where $s$ is a section of the associated $V$-bundle over the formal disc. The action map $\mathcal{T}_{G,V} \to V\brK$ is given by using the trivialization over the punctured disc to consider $s$ as an element of $V\brK$.

\begin{defn}
    The \emph{BFN Steinberg variety} is the fiber product
    $$\mathcal{T} \times_{V\brK} \mathcal{T}$$
    induced by the action map $\mathcal{T} \to V\brK$.
\end{defn}
\begin{rem}
    This is an analog of the usual Steinberg variety $T^*G/B \times_{\mathcal{N}} T^*G/B$.
\end{rem}

\begin{defn}
    The \emph{moduli of triples} is the subspace defined as the preimage 
    $$\mathcal{R}_{G, V} := \mu^{-1}(V\brO) \subseteq \mathcal{T}_{G, V}, $$ 
    where $\mu : \mathcal{T}_{G, V} \to V\brK$ is the action map.
\end{defn}

By definition, $\mathcal{R}_{G, V}$ consists of elements in $\mathcal{T}_{G, V}$ such that the section $s$ is mapped to a formal power series (as opposed to Laurent series) under the trivialization $\phi$. The advantage of considering $\mathcal{R}_{G, V}$ is that the BFN Steinberg variety equipped with the $G\brK$-action can be identified as stacks with $\mathcal{R}_{G,V}$ equipped with the $G\brO$-action. The fact that $G\brO$ is pro-algebraic (in particular, much smaller than $G\brK$) is leveraged to define the commutative convolution product structure on equivariant (Borel--Moore) homology of $\mathcal{R}_{G,V}$ through the \emph{convolution diagram}, see \cite[3(i, ii, iii)]{BFN2}.

\begin{defn}[{\cite[Definition 3.13]{BFN2}}]
    The \emph{BFN Coulomb branch} is 
    $$\mathcal{M}_C(G, V):= \mathrm{Spec} \ H_*^{G\brO}(\mathcal{R}_{G, V}; k), $$ 
    where $H_*^{G\brO}(\mathcal{R}_{G, V}; k)$ is endowed with the convolution product structure.
\end{defn}

The group $\mathbb{G}_m = \mathbb{C}^*$ acts on $\mathcal{R}_{G, V}$ as automorphisms of the (punctured) formal disc, and the equivariant (Borel--Moore) homology 
\begin{equation*}
    \mathcal{A}_{\hbar}(G,V) := H_*^{G\brO \rtimes \mathbb{G}_m}(\mathcal{R}_{G, V}; k)
\end{equation*}
with the convolution product structure defines a noncommutative deformation of the Coulomb branch algebra $H_*^{G\brO}(\mathcal{R}_{G, V}; k)$, inducing a Poisson structure on $\mathcal{M}_C(G, V)$. 

Deformations and resolutions of Coulomb branches may be obtained by introducing an additional torus action $T_F \to \mathrm{GL}(V)$ commuting with the $G$-action, or more generally an extension of $T_F$ by $G$ which fits into a short exact sequence
\begin{equation*}
    1 \longrightarrow G \longrightarrow \tilde{G} \longrightarrow T_F \longrightarrow 1,
\end{equation*}
see \cite[3(viii, ix)]{BFN2}. The point is that if $V$ is a $\tilde{G}$ representation, $\mathcal{M}_C(G, V)$ is a Hamiltonian reduction of $\mathcal{M}_C(\tilde{G}, V)$ under the dual group $T_F^{\vee}$-action, and $\mathcal{A}_{\hbar}(G,V)$ is obtained from $\mathcal{A}_{\hbar}(\tilde{G},V)$ by quantum Hamiltonian reduction. The deformation space is given by $\mathrm{Lie}(T_F)$. On the other hand, upon choosing a character of $T_F^{\vee}$, one can define a (partial) resolution of $\mathcal{M}_C(G, V)$, which has been conjectured to give rise to a symplectic resolution in \cite{BFN2} and verified in \cite{bellamy}.

\subsection{Symplectic duality}\label{ssec:symp-duality}
Symplectic duality posits that conical Hamiltonian symplectic resolutions arise in pairs, where certain structures are interchanged for dual pairs. Depending on what the structure of concern is, there are distinct formulations of the symplectic duality phenomena, as surveyed in \cite[Section 5]{kamnitzer-survey}. In this subsection, we explain the formulation of the \emph{quantum Hikita conjecture} which identifies the quantum $D$-module and the $D$-module of twisted traces for dual pairs.

\subsubsection{Quantization exact sequence and cohomology exact sequence}\label{sssec:eq-kah-ses}

Consider a conical Hamiltonian symplectic resolution $X$ with a $T \times \mathbb{G}_m$-action. We may construct two short exact sequences of vector spaces involving the canonical quantization of $X$ and the $T \times \mathbb{G}_m$-equivariant cohomology of $X$
respectively, which are interchanged for symplectically dual pairs.

\begin{defn}
    The \emph{quantization exact sequence} is the short exact sequence
    \begin{center}
        \begin{tikzcd}
    0 \rar & H_2(X; k)\oplus k\hbar \rar & \mathcal{A}_0^2 \rar & \mathfrak{t} \rar &0
\end{tikzcd}
    \end{center}
    of $k$-linear Lie algebras, where $\mathcal{A}^2$ (hence $\mathcal{A}_0^2$) has Lie bracket given by the commutator $\hbar^{-1}[-,-]$.   
\end{defn}
\begin{rem}
The commutator $\hbar^{-1}[-,-]$ is well-defined without inverting $\hbar$, since by the definition of a quantization, we have $[\mathcal{A}, \mathcal{A}] \subseteq \hbar \mathcal{A}$ and $\hbar$ acts on $\mathcal{A}$ in a torsion-free way (the sheaf of quantizations, for which $\mathcal{A}$ is the global sections, is a flat deformation of $\mathcal{O}_{\mathcal{X}}$)).
\end{rem}

\begin{defn}
    The \emph{cohomology exact sequence} is the short exact sequence
    \begin{center}
        \begin{tikzcd}
  0 \rar & H^2_{T\times \mathbb{G}_m} (\mathrm{pt};k)\rar & H^2_{T \times \mathbb{G}_m ; k} (X) \rar &  H^2(X;k)\rar & 0  
    \end{tikzcd}
    \end{center}
    of $k$-vector spaces, where $H^2_{T \times \mathbb{G}_m}(pt ; k) \cong k \lambda_1 \oplus \cdots \oplus k \lambda_{\mathrm{rank} T} \oplus k \hbar$ is additively generated by the equivariant parameters.
\end{defn}

For symplectically dual pairs $(X, X^!)$, we expect the quantization exact sequence and the cohomology exact sequence to be interchanged, with isomorphisms even over integers $H^2(X^!;\mathbb{Z}) \cong \mathfrak{t}_{\mathbb{Z}}$ and $H^2_{T^!}(\mathrm{pt};\mathbb{Z}) \cong (\mathfrak{t}^!)^*_{\mathbb{Z}} \cong H_2(X;\mathbb{Z})$. Both of the exact sequences are (non-canonically) split: for the former, see \cite[Remark 2.2]{KMP21}; for the latter, it follows from equivariant formality.

\subsubsection{Equivariant and K\"ahler roots}\label{sssec:eq-kah-roots}
Consider a conical Hamiltonian symplectic resolution $X$ with a $T \times \mathbb{G}_m$-action. We may construct finite collections of weights in $X^\bullet(T)$ and curve classes in $H_2(X; \mathbb{Z})$ dubbed the \emph{equivariant roots} and the \emph{K\"ahler roots} respectively, which are interchanged for symplectically dual pairs.

We first define the equivariant roots. Denote the canonical quantization as $\mathcal{A}$. We denote $\mathcal{A}^+ \subseteq \mathcal{A}$ to be the two-sided ideal of elements with positive conical grading.

\begin{defn}
    The \emph{equivariant roots} of $X$ are the weights of $\mathcal{A}^+ / (\mathcal{A}^+ \cdot \mathcal{A}^+)$ as a $T$-representation. The set of equivariant roots of $X$ are denoted $\Sigma_{eq}:= \Sigma_{eq}(X)$.
\end{defn}

A choice of a cocharacter $\sigma : \mathbb{G}_m \to T$ determines a distinguished subset of $\Sigma_{eq}$:
\begin{defn}
    The \emph{positive equivariant roots} of $X$ with respect to a cocharacter $\sigma: \mathbb{G}_m \to T$ are
    \begin{equation}
        \Sigma_{eq,+} = \Sigma_{\sigma, +} := \{ \lambda \in \Sigma_{eq} : \langle \sigma, \lambda \rangle > 0 \}.
    \end{equation}
\end{defn}
The positive equivariant roots generate an additive semigroup denoted by $\mathbb{Z}_{\ge 0} \Sigma_{eq,+}$.

We now define the K\"ahler roots. The definition of K\"ahler roots depends on a (hypothetical) structure theorem for the equivariant quantum cohomology of symplectic resolutions conjectured by Okounkov. 

Denote by $\Lambda' := k [z^\alpha : \alpha \in H_2^{\mathrm{eff}}(X;\mathbb{Z})]$ to be the \emph{polynomial Novikov ring}, generated over $k$ by symbols $z^\alpha$ where $\alpha$ indexes the additive semigroup of effective curve classes in $H_2(X;\mathbb{Z})$. Let $\hat{\Lambda}$ be the completion of $\Lambda'$ at the maximal ideal generated by $z^\alpha$. 

Recall that the \emph{(equivariant) quantum cohomology} of $X$ is the $k$-vector space 
\begin{equation}
    QH^*(X):=H^*_{T \times \mathbb{G}_m}(X) \otimes \hat{\Lambda}
\end{equation}
endowed with a graded commutative product $\star$ which is constructed using genus $0$, $3$-pointed equivariant Gromov--Witten invariants. Denote by $\hbar$ the generator of $H^2_{\mathbb{G}_m}(pt)$, which represents the conical weight of the holomorphic symplectic form.

\begin{conj}\label{conj:structure-of-QH}(\cite[\S 2.3.4]{okounkov15})
    For a conical Hamiltonian symplectic resolution $X$ with $T\times\mathbb{G}_m$-action, there exists a finite set $\Sigma_{Kah, +} \subseteq H_2(X;\mathbb{Z})$ spanning the semigroup $H_2^{\mathrm{eff}}(X;\mathbb{Z})$ of effective curve classes, such that for each $ \alpha \in \Sigma_{Kah, +}$ there exists an operator $S_\alpha \in \mathrm{End}(H^*_{T \times \mathbb{G}_m}(X))$ such that the quantum product by any $u \in H^2_{T \times \mathbb{G}_m}(X)$ acts on $QH^*_{T \times \mathbb{G}_m}(X)$ as 
    \begin{equation}
        u  \star \ \cdot = u \smile \cdot + \hbar \sum_{\alpha \in \Sigma_{Kah, +}} \langle \alpha, \bar{u} \rangle \frac{z^{\alpha}}{1-z^\alpha} S_\alpha(\cdot).
    \end{equation}
    Here the action of $S_\alpha$ is extended $\hat{\Lambda}$-linearly to $QH^*_{T \times \mathbb{G}_m}(X)$ and $\bar{u}$ is the image of $u$ under the projection $H^2_{T \times \mathbb{G}_m}(X) \to H^2(X)$.
\end{conj}
In particular, the conjecture implies that the quantum cohomology can be defined over $H^*_{T \times \mathbb{G}_m}(X) \otimes \Lambda$, where $\Lambda = \Lambda' [(1-z^\alpha)^{-1}: \alpha \in \Sigma_{Kah, +}] \subseteq \hat{\Lambda}$ is the localization which is much more restricted than a general power series in $z^\alpha$.

\begin{defn}
    Assuming \cref{conj:structure-of-QH}, the corresponding elements of the finite set $\Sigma_{Kah, +}$ are the \emph{positive K\"ahler roots}, and an element of $\Sigma_{Kah, +} \cup -\Sigma_{Kah, +} \subseteq H_2(X;\mathbb{Z})$ is called a \emph{K\"ahler root}.
\end{defn}

\begin{rem}
    Note that the notion of a positive K\"ahler root crucially depends on the submonoid $H_2^{\mathrm{eff}}(X;\mathbb{Z}) \subseteq H_2(X;\mathbb{Z})$ of effective curve classes. For instance, given two birational symplectic resolutions $X_1 \to Y$, $X_2 \to Y$ for a fixed affine $Y$, these monoids may differ as the K\"ahler forms on $X_1, X_2$ may differ. In other words, $\Sigma_{Kah, +}$ implicitly depends on the choice of the resolution $X \to Y$, which we fix from the beginning as in our definition of symplectic resolutions.
\end{rem}

For symplectically dual pairs $(X, X^!)$, symplectic duality predicts that the equivariant and the K\"ahler roots are interchanged.

\subsubsection{$D$-module of twisted traces and the quantum $D$-module}\label{sssec:eq-kah-dmodules}

The quantum Hikita conjecture identifies two $D$-modules associated to a conical Hamiltonian symplectic resolution for a dual pair. We first describe the relevant rings of differential operators for both constructions, over which we will define the respective $D$-modules.

Fix a conical Hamiltonian symplectic resolution $(X, \omega)$ with its $T \times \mathbb{G}_m$-action, and choose a cocharacter $\sigma : \mathbb{G}_m \to T$. Denote the corresponding positive equivariant roots by $\Sigma_{eq,+}$. Also, assuming Conjecture \ref{conj:structure-of-QH}, write the set of positive K\"ahler roots as $\Sigma_{Kah, +}$.

\begin{defn}\label{defn:Dring-eq}
    Define the \emph{equivariant ring of differential operators} to be the graded $k[\hbar]$-algebra
    \begin{equation}
        R_{eq} : =   k [z^\mu] \otimes \mathrm{Sym} \ \mathcal{A}^2_0,
    \end{equation}
     where $z^\mu$ are formal variables of degree $0$ indexed by $\mu \in \mathbb{Z}_{\ge 0} \Sigma_{eq,+}$, subject to the relation
     \begin{equation}
      a z^\mu = z^\mu (a + \hbar \langle \mu, \bar{a} \rangle ),
     \end{equation}
     where $\bar{a} \in \mathfrak{t}$ denotes the image of $a \in \mathcal{A}_0^2$ under the projection map in the quantization exact sequence.
\end{defn}

Here, the $k[\hbar]$-module structure on $R_{eq}$ comes from the quantization exact sequence. Although $R_{eq}$ is not commutative, one can show that there exists a well-defined localization
$R_{eq, reg} := R_{eq} [(1-z^\mu)^{-1}]$ obtained by inverting $(1-z^\mu)$ for $\mu \in \Sigma_{eq,+}$, see \cite[Lemma 3.4]{KMP21}.

Now we define the $D$-module of twisted traces from the universal quantization $\mathcal{A}$ of $X$.

\begin{defn}\label{defn:Dmod-eq}
    The \emph{$D$-module of twisted traces} for the $X^\bullet(T) \times \mathbb{Z}$-graded algebra $\mathcal{A}$ is the graded left $R_{eq}$-module
    \begin{equation}
        M_{eq} = \mathscr{B}(\mathcal{A}) :=  k[z^\mu : \mu \in \mathbb{Z}_{\ge 0} \Sigma_{eq,+}] \otimes \mathcal{A}_0 / \{ 1 \otimes ab - z^{\lambda} \otimes ba: a \in \mathcal{A}_\lambda, b \in \mathcal{A}_{-\lambda} \},
    \end{equation}
    where the action of $z^\mu \otimes a \in  k[z^\mu] \otimes \mathcal{A}_0^2 \subseteq R_{eq}$ is given by
    \begin{equation}
        (z^\mu \otimes a) \cdot (z^\lambda \otimes b) = z^{\lambda + \mu} \otimes (a + \hbar \langle \lambda, \bar{a} \rangle )b  .
    \end{equation}
\end{defn}
The fact that the $D$-module of twisted traces is well-defined is verified in \cite[Proposition 3.5]{KMP21}. By inverting $(1-z^\mu)$ for $\mu \in \Sigma_{eq,+}$, we obtain a $R_{eq, reg}$-module $M_{eq, reg} := R_{eq, reg} \otimes_{R_{eq}} M_{eq}$.

The module $\mathscr{B}(\mathcal{A})$ is considered as a \emph{quantum} (i.e. $z$-parameter) deformation of the $B$-algebra $B(\mathcal{A})$. It does not carry a natural algebra structure, unless $\mathcal{A}$ itself is commutative.

\begin{rem}
In \cite{KMP21}, the object is referred to as the $D$-module of \emph{graded traces}, and our terminology follows \cite{etingof-stryker}, considering the $X^\bullet(T)$-grading on $\mathcal{A}$ as inducing the ``twisting.''
\end{rem}

The counterpart of the ring of differential operators on the quantum cohomology side is the following.
\begin{defn}\label{defn:Dring-quantum}
    The graded $k[\hbar]$-algebra
    \begin{equation}
        R_{Kah} := k [ z^\alpha] \otimes \mathrm{Sym} H^2_{T \times \mathbb{G}_m} (X)
    \end{equation}
is the \emph{K\"ahler ring of differential operators}, where $z^\alpha$ is the formal variable of degree $0$ indexed by $\alpha \in \mathbb{Z}_{\ge 0} \Sigma_{Kah, +}$, subject to the relation
\begin{equation}
x z^\alpha = z^\alpha (x + \hbar \langle \alpha, \bar{x} \rangle ),
\end{equation}
where $\bar{x} \in H^2(X)$ denotes the image of $x \in H^2_{T \times \mathbb{G}_m}(X)$ under the projection map in the cohomology exact sequence.
\end{defn}
Similar to the definition of $R_{eq, reg}$, by inverting $(1-z^\alpha)$ for $\alpha \in \Sigma_{Kah, +}$, we obtain $R_{Kah, reg}$.

The quantum $D$-module is defined using the equivariant quantum cohomology of $X$, whose definition in the present form depends on assuming Conjecture \ref{conj:structure-of-QH}.

\begin{defn}\label{defn:Dmod-quantum}
    The \emph{(specialized) quantum $D$-module} for the conical Hamiltonian symplectic resolution $X$ is the graded left $R_{Kah,reg}$-module
    \begin{equation}
        M_{Kah, reg} :=  \Lambda \otimes H^*_{T \times \mathbb{G}_m}(X) \cong k[z^\alpha,(1-z^\alpha)^{-1}] \otimes H^*_{T \times \mathbb{G}_m}(X),
    \end{equation}
    where the actions of $z^\alpha \otimes 1, 1 \otimes x \in k[z^\alpha ] \otimes H^2_{T \times \mathbb{G}_m}(X) \subseteq R_{Kah}$ are given by
    \begin{equation}
        (z^\alpha \otimes 1) \cdot (z^\beta \otimes y) =  z^{\alpha + \beta} \otimes y , \quad (1 \otimes x) \cdot (z^\beta \otimes y) =  (z^\beta \otimes 1) (x \star y  + 1 \otimes \hbar \langle \beta, \bar{x} \rangle y)
    \end{equation}
    Here $\star$ denotes the equivariant quantum product. The action on $(1-z^\alpha)^{-1} \otimes 1 $ is determined as one would define formally for $(1-z)^{-1} = 1 + z + z^2 + \cdots$.
\end{defn}

Note that the quantum $D$-module only exists as a $R_{Kah, reg}$-module, as the quantum product $\star$ involves denominators of the form $(1-z^{\alpha})^{-1}$ for $\alpha \in \Sigma_{Kah, +}$, see \cref{conj:structure-of-QH}. The quantum $D$-module is closely related to the object known as the \emph{quantum connection}, where the action of $x$ is interpreted as the ``covariant derivative'' $\hbar \partial_{\bar{x}} + x \star$ where $\partial_{\bar{x}}$ acts on the variable $z^\beta$ by $\partial_{\bar{x}} z^{\beta} = \langle \beta, \bar{x} \rangle z^\beta$.

\begin{rem}{cf. {\cite[Remark 4.3]{KMP21}}}\label{rem:hbar-vs-t}
    This object is denoted the \emph{specialized} quantum $D$-module, as it identifies the two \emph{a priori} unrelated equivariant parameters in the (usual) construction of the quantum $D$-module. Namely, we identify the \emph{loop parameter} relevant for the action of $R_{Kah}$ on $M_{Kah, reg}$ with the $\mathbb{G}_m(\mathbb{C})$-equivariant parameter from the conical action on $X$, and denote both by $\hbar$. In terms of the enumerative geometry of maps $u: \mathbb{CP}^1 \to X$, the former loop parameter arises from the symmetry of the source curve, while the latter $\mathbb{C}^\times$-equivariant parameter arises from the symmetry of the target symplectic variety. This specialization drastically simplifies the quantum $D$-module, but we do not know a formulation of 3D mirror symmetry which gives independent interpretations of the two parameters.
\end{rem}

For symplectically dual pairs $(X, X^!)$, we expect that the equivariant and K\"ahler ring of differential operators are interchanged, identifying the $D$-module of twisted traces and the (specialized) quantum $D$-module. We spell this conjectural correspondence more precisely in the next subsection.

\subsection{Quantum Hikita conjecture}\label{ssec:quantum-hikita}

Here, we survey the formulation of symplectic duality via the quantum Hikita conjecture, following \cite[Section 5]{KMP21}. In this subsection, we assume that $k = \mathbb{C}$ and that the symplectic resolutions $X \to Y$ and the torus acting on $X$ are defined over $\mathbb{C}$. Assume that $X$ and $X^!$ are symplectically dual pairs of conical Hamiltonian symplectic resolutions, equipped with actions of algebraic tori $T$ and $T^!$. We denote the equivariant/Kahler roots, ring of differential operators, and the $D$-modules constructed for $X^!$ with the superscript $!$. 

\begin{rem}\label{rem:s-dual-is-empirical}
Note that the notion of symplectically dual pairs is empirical in nature, and we do not have a general theory for producing a dual pair; nevertheless, many cases are known (see e.g. \cite[Section 5]{kamnitzer-survey} and references therein), with the largest class given by Higgs/Coulomb branches of a gauge theory $(G, V)$.
\end{rem}

The Hikita conjecture, originally formulated as \cite[Conjecture 1.3]{hikita}, states that the cohomology ring of $X^!$ is isomorphic to the coordinate ring of the (scheme-theoretic) $T$-fixed locus of $Y$. 

\begin{conj}[Hikita conjecture]\label{conj:hikita} Let $X, X^!$ be symplectically dual and denote by $Y$ the affinization of $X$. Then there exists an isomorphism of graded algebras
\begin{equation}
    \mathcal{O}(Y^T) \cong H^*(X^!; k),
\end{equation}
where $Y^T$ denotes the scheme-theoretic fixed locus of $Y$ under the $T$-action, with grading induced from the scaling $\mathbb{G}_m$-action.
\end{conj}

To better relate the Hikita conjecture with its equivariant and quantum generalizations, we take a different perspective. Namely, the coordinate ring $\mathcal{O}(Y^T)$ can be described in terms of the $B$-algebra construction.

\begin{defn}\label{defn:B-algebra}
    Given a $X^\bullet(T)$-graded associative algebra $A$ and a choice of a cocharacter $\sigma \in X_\bullet(T)$, the \emph{$B$-algebra} of $A$ is
    \begin{equation}
        B(A) = A_0 / \sum_{\langle \sigma, \lambda \rangle > 0} A_\lambda A_{-\lambda}.
    \end{equation}
\end{defn}

In the commutative case, where $A$ is the coordinate ring of an affine variety with a $T$-action, the $0$-weight part corresponds to the coordinate ring of the orbit space (affine GIT quotient). Quotienting out by $A_\lambda A_{-\lambda}$ removes the non-trivial orbits in the orbit space, so that $\mathrm{Spec}B(A)$ is the (scheme-theoretic) fixed locus. As $X \to Y$ is the affinization map, we know that $\mathcal{O}(X) = \mathcal{O}(Y)$ as commutative algebras. Therefore, Conjecture \ref{conj:hikita} can be restated as the isomorphism
\begin{equation}
    B(\mathcal{O}(X)) \cong H^*(X^!; k).
\end{equation}

The equivariant Hikita conjecture (the Hikita--Nakajima conjecture, cf. \cite[Conjecture 8.9]{KTWWY}) extends the Hikita conjecture to a conjectural isomorphism between the $B$-algebra of the universal quantization $\mathcal{A}$ of $\mathcal{O}(X)$ and the equivariant cohomology of $X^!$:

\begin{conj}[Hikita--Nakajima conjecture]\label{conj:hikita-nakajima}
Let $X, X^!$ be symplectically dual, and denote by $\mathcal{A}$ the universal quantization of $X$. Then there exist
\begin{itemize}
    \item an isomorphism between the quantization exact sequence
    \begin{center}
        \begin{tikzcd}
    0 \rar & H_2(X)\oplus k\hbar \rar & \mathcal{A}_0^2 \rar & \mathfrak{t} \rar &0
\end{tikzcd}
    \end{center}
    and the cohomology exact sequence
     \begin{center}
        \begin{tikzcd}
  0 \rar & H^2_{T^!\times \mathbb{G}_m} (\mathrm{pt})\rar & H^2_{T^! \times \mathbb{G}_m} (X^!) \rar &  H^2(X^!)\rar & 0.  
    \end{tikzcd}
    \end{center}
of $k$-vector spaces, together with
    \item an isomorphism of graded and commutative $H^2_{T^!\times \mathbb{G}_m} (\mathrm{pt}) \cong k[H_2(X)][\hbar]$-algebras \begin{equation}
    B(\mathcal{A}) \cong H^*_{T^! \times \mathbb{G}_m} (X^!;k),
    \end{equation}
    which specializes to the isomorphism predicted by the Hikita conjecture.
\end{itemize}    
\end{conj}

In particular, Conjecture \ref{conj:hikita-nakajima} predicts isomorphisms $H^2(X) \cong \mathfrak{t}^!$ and $\mathfrak{t} \cong H^2(X^!)$ (the identification of equivariant and K\"ahler parameters). Also note that $B(\mathcal{A})$ is not \emph{a priori} commutative, but the existence of the isomorphism with the commutative $H^*_{T^! \times \mathbb{G}_m} (X^! ;k)$ implies that the ring structure is in fact commutative.

\begin{rem}
    Note that the universal quantization $\mathcal{A}$ is constructed from the commutative $\mathcal{O}(X)$ in two steps: (i) one first takes the universal deformation $\mathcal{X}$ of $X$, whose deformation base is $H^2(X)$, and (ii) one takes the $k[\![\hbar]\!]$-linear quantization of $\mathcal{O}(\mathcal{X})$. These are reflected in the introduction of (i) $T^!$-equivariant parameters $H^2_{T^!}(\mathrm{pt}) \cong (\mathfrak{t}^!)^*$ and the (ii) $\mathbb{G}_m$-equivariant parameter $H^2_{\mathbb{G}_m}(\mathrm{pt}) \cong k \hbar$ on the dual side, respectively.
\end{rem}

The \emph{quantum Hikita conjecture} introduced by \cite[Conjecture 5.1]{KMP21} extends the Hikita--Nakajima conjecture further to include the data of roots and $D$-modules.

\begin{conj}[Quantum Hikita conjecture]\label{conj:hikita-quantum} Let $X, X^!$ be symplectically dual. Then there exist
\begin{itemize}
    \item a bijection of the roots $\Sigma_{eq,+} \cong \Sigma_{Kah, +}^!$, 
    \item a compatible isomorphism of the ring of differential operators $R_{eq} \cong R_{Kah}^!$,
    \item and a compatible isomorphism of the $D$-modules $M_{eq, reg} \cong M_{Kah, reg}^!$ taking $1 \in \mathcal{A}_0^0$ to $1 \in H^0(X^!)$,
\end{itemize}
extending the isomorphisms from the Hikita--Nakajima conjecture.
\end{conj}

The compatibility of the isomorphisms refers to the following. First, there is an identification of the equivariant roots and the K\"ahler roots of $X$, $X^!$ respectively, inducing an isomorphism of monoids $\mathbb{Z}_{\ge 0} \Sigma_{eq, +} \cong \mathbb{Z}_{\ge 0} \Sigma_{Kah, +}^!$. Together with the isomorphism $\mathcal{A}_0^2 \cong H^2_{T^! \times \mathbb{G}_m}(X^!)$ depicted by the equivariant Hikita conjecture, we have an isomorphism of graded $k[\hbar]$-algebras
\begin{equation}
R_{eq} :=  k[z^\mu : \mu \in \mathbb{Z}_{\ge 0} \Sigma_{eq, +}] \otimes \mathrm{Sym} \mathcal{A}_0^2 \cong k [z^\alpha : \alpha \in \mathbb{Z}_{\ge 0 } \Sigma_{Kah, +}^!] \otimes \mathrm{Sym}H^2_{T^! \times \mathbb{G}_m}(X^!) =: R_{Kah}^!.
\end{equation}
After localization, this isomorphism implies the isomorphism of graded algebras
\begin{equation}
    R_{eq, reg} \cong R_{Kah, reg}^!.
\end{equation}
Lastly, we should have an isomorphism
\begin{equation*}
    M_{eq, reg} := \mathscr{B}(\mathcal{A}) [(1-z^\mu)^{-1}] \cong  \Lambda \otimes H^*_{T^! \times \mathbb{G}_m}(X^{!}) =: M_{Kah, reg}^! 
\end{equation*}
of graded $R_{eq, reg} \cong R_{Kah, reg}^!$-modules.

\begin{rem}\label{rem:hikita-nakajima}
    If we take the $z \rightarrow 0$ limit on the $D$-modules, the quantum Hikita conjecture specializes to an isomorphism $B(\mathcal{A}) \cong H^*_{T^! \times \mathbb{G}_m}(X^{!})$ as graded \emph{modules} over $\mathrm{Sym} \mathcal{A}_0^2 \cong \mathrm{Sym}H^2_{T^! \times \mathbb{G}_m}(X^!)$, which is \emph{a priori} weaker than the algebra isomorphism expected by the Hikita--Nakajima conjecture. However, in examples for which the quantum Hikita conjecture is proven in \cite{KMP21}, $H^*_{T^! \times \mathbb{G}_m}(X^{!})$ is generated by degree $2$ classes.
\end{rem}

\begin{thm}[{\cite[Theorem 6.13, Theorem 7.17]{KMP21}}]
    The quantum Hikita conjecture holds for hypertoric varieties and the Springer resolution.
\end{thm}

\section{Quantum Hikita conjecture mod $p$: Statement}\label{sec:quantum-hikita-modp}

In this section, we provide the statement of our mod $p$ 3D mirror symmetry conjecture, inspired by the quantum Hikita proposal of \cite{KMP21}, as an isomorphism of the mod $p$ quantum $D$-module and the mod $p$ $D$-module of twisted traces for symplectically dual pairs, \emph{together with the action of Frobenius-linear covariantly constant operators}.

For the mod $p$ quantum $D$-module, these Frobenius-linear operators are given by the \emph{quantum Steenrod operators} \cite{seidel-wilkins} (or equivariant generalizations thereof, \cite{Lee23b}). We emphasize that for the quantum $D$-module the symplectic resolution itself is still defined over $\mathbb{C}$, and the positive characteristic only enters through the coefficients of quantum cohomology algebra.

For the mod $p$ $D$-module of twisted traces, these Frobenius-linear operators are given by the multiplication operators arising from \emph{Frobenius-constant quantizations} of the underlying symplectic singularity. Such Frobenius-constant quantizations exist for a large class of symplectic singularities, including Springer resolutions \cite{bezrukavnikov-mirkovic-rumynin} and all BFN Coulomb branches \cite{lonergan}.

We fix $k$ an algebraically closed field of characteristic $p$, which is assumed to be $\gg 0$ if necessary.

\begin{rem}
    Note that the constructions introduced in this section, e.g. that of quantum Steenrod operators, can be made for any coefficients of positive characteristic, and there is no need to take $k = \mathbb{F}_p$. However, there may exist a finite set of primes over which the relevant objects (such as the quantum $D$-module) cannot be defined, which is why we fix a generic $p$ (cf. the discussion in \cref{sssec:exmp-springer}).
\end{rem}

Denote by $\mu_p := \mu_p(\mathbb{C}) \subseteq \mathbb{G}_m(\mathbb{C})$ the multiplicative group of complex $p^\mathrm{th}$ roots of unity. By choosing a generator $\zeta = e^{2\pi i/p} \in \mu_p$, we obtain (and fix) an isomorphism $\mu_p \cong \mathbb{Z}/p\mathbb{Z}$ of abelian groups.

We now describe the candidate descriptions for the Frobenius-linear actions mentioned above. 

\subsection{Quantum Steenrod operations}\label{ssec:oper-qst}

In characteristic $p$, the quantum $D$-module admits an action of the (Frobenius-twisted) mod $p$ quantum cohomology algebra, known as the quantum Steenrod operations. These are (Frobenius-)linear operators that deform the usual action of Steenrod powers in mod $p$ singular cohomology by taking counts of certain holomorphic spheres with a $\mathbb{Z}/p$-symmetry into account. 

\subsubsection{A heuristic description}\label{sssec:qst-heuristic}
We discuss the quantum Steenrod operators of \cite{seidel-wilkins}, which reformulates the quantum Steenrod operations of \cite{Fuk97} as linear endomorphisms of the mod $p$ quantum cohomology. We also need their torus-equivariant generalizations introduced in \cite{Lee23b} for applications to conical Hamiltonian symplectic resolutions.

The actual definition of quantum Steenrod operators involves constructions in symplectic enumerative geometry, due to technical issues in curve-counting with positive characteristic coefficients; in particular, an algebro-geometric definition is not currently available.  Here we give a heuristic description of how the definition of these operators would look like \emph{in a hypothetical situation} where these foundational issues do not exist, and point out the actual difficulties. For the actual construction in symplectic geometry which allows one to overcome such difficulties, we refer to \cite{seidel-wilkins} (for the original definition) and \cite{Lee23a, Lee23b} (for the torus-equivariant versions).

Fix $C =\mathbb{P}^1 = \mathbb{C} \cup \{ \infty \}$ to be a curve with the following $(p+2)$ distinguished marked points
\begin{equation}
z_0 = 0, \ z_1 = \zeta = e^{2\pi i / p}, \dots, z_p = \zeta^p = 1,\  z_\infty = \infty.
\end{equation}
We consider parametrized maps from $C$ to a smooth (quasi)projective complex variety $X$.

Let $\tau$ be the rotation of the curve $C$ by the angle of $2\pi/p$ which fixes $z_0$ and $z_\infty$, then, it cyclically permutes the marked points as
\begin{equation}                                  
    \tau : (C, z_0, z_1 , z_2 \dots, z_p, z_\infty) \mapsto (C, z_0, z_2, z_3, \dots, z_1, z_\infty).
\end{equation}
Using $\tau$, one defines an action of the cyclic group $\sg$ on $C$, whose generator is identified with $\zeta = e^{2\pi i/p}$. For $\eta = \zeta^j \in \mathbb{Z}/p$, we denote by $\tau(\eta) = \tau^j$ the corresponding automorphism of $C$.

There is also an action of $\mathbb{Z}/p$ which cyclically permutes the product $X^p$:
\begin{align}\label{eqn:X^p-cyclic-permutation}
    \sigma_{X} :& \ \mathbb{Z}/p \times  X \times X^p \times X \to X \times X^p \times X, \\
    & [\eta; (x_0; x_1, x_2, \dots, x_{p}; x_\infty)] \mapsto (x_0; x_{1+\eta}, x_{2+\eta}, \dots, x_{p+\eta}; x_\infty),
\end{align}
where the numbers $k+\eta$ are read mod $p$.

The moduli space of parametrized maps $u: C  \cong \mathbb{P}^1 \to X$,
\begin{equation}
\mathcal{M}_\alpha := \{ u : C \to X : u_*[C] = \alpha \in H_2(X;\mathbb{Z}) \},
\end{equation}
carries an evaluation map
\begin{align}
    \mathrm{ev}: \mathcal{M}_\alpha &\to X \times X^p \times X \\
    (u: C \to X) &\mapsto \left( u(z_0); u(z_1), \dots, u(z_p); u(z_\infty) \right)
\end{align}
that is $\mathbb{Z}/p$-equivariant with respect to the action on $\mathcal{M}_\alpha$ induced by source rotation $\tau$ and the action $\sigma_X$ on the target. The moduli space $\mathcal{M}_\alpha$ is noncompact in general, and must be suitably compactified by stable maps as is well known (see, e.g. \cite[Chapter 6]{MS12}). When the target $X$ is non-compact, the properness of the stable maps and the evaluation map $\mathrm{ev}$ require additional hypotheses on $X$ such as the existence of the conical $\mathbb{G}_m$-action, see e.g. \cite{RZ23}.  Let us assume that $\mathcal{M}_\alpha$ admits a $\mathbb{Z} / p$-equivariant compactification $\overline{\mathcal{M}}_\alpha$ (by stable maps) over which $\mathrm{ev}$ extends to a $\mathbb{Z} / p$-equivariant proper map, and continue to denote the extension by $\mathrm{ev}$ (such an extension is indeed provided in \cite{Lee23a, Lee23b}).

Fix a cohomology class $x \in H^*(X;k)$, and consider its image
$x_{eq} := x \otimes x \otimes \cdots \otimes x \in H^*_{\mu_p}(X^p ; k)$
under the (non-linear) map that sends $x \in H^*(X;k)$ to its $p$-fold tensor power. \emph{Assuming the existence of a suitable equivariant virtual fundamental cycle of $\overline{\mathcal{M}}_\alpha$ over $k$-coefficients}
$$[\overline{\mathcal{M}}_\alpha]^{vir}_{\mu_p} \in H_{-*}^{\mu_p} (\overline{\mathcal{M}}_\alpha ; k),$$ 
one would define the corresponding ``correlator'' by the push-pull diagram
\begin{center}
    \begin{tikzcd}
        \overline{\mathcal{M}}_\alpha \rar["\mathrm{ev}"] \dar["\mathrm{proj}"] & X \times X^p \times X \\ \mathrm{pt} & 
    \end{tikzcd}
\end{center}
as
\begin{equation}
\langle x_0 ; x_{eq}; x_\infty \rangle_{0, \alpha} := (\mathrm{proj})_* \left( [\overline{\mathcal{M}}_\alpha]^{vir}_{\mu_p} \cap \mathrm{ev}^* (x_0 \otimes x_{eq} \otimes x_\infty)  \right) \in H^*_{\mu_p}(\mathrm{pt}) \cong k [\![t, \theta]\!].
\end{equation}
Here, $H^*_{\mu_p}(\mathrm{pt}) \cong k [\![t, \theta]\!]$ is the group cohomology of $\mu_p(\mathbb{C})$ over $k$, the graded polynomial algebra on two commuting variables $t, \theta$ of degree $|t| = 2$ and $|\theta| = 1$. (The notation $k[\![t,\theta]\!]$ refers to the \emph{$\mathbb{Z}$-graded} completion, and since $t, \theta$ carry nontrivial degrees, $k[\![t,\theta]\!] \cong k [t] \otimes\Lambda(\theta)$.)

Now recall that the quantum cohomology ring is an algebra over the (completed) \emph{Novikov ring} $\hat{\Lambda} = k[\![z^\alpha : \alpha \in H_2^{\mathrm{eff}}(X;\mathbb{Z})]\!]$ (cf. \cref{sssec:eq-kah-roots}). The quantum Steenrod operators are defined as linear operators on quantum cohomology over $k$, whose structure constants are given by the correlators above. For the definition, we fix a $k$-linear basis $B$ of $H^*(X;k)$ (and the definition is independent of the basis by linearity):
\begin{defn}[A \emph{pseudo}-definition]
    Fix $x \in H^*(X;k)$. The \emph{quantum Steenrod operator} $\Sigma_x : H^*(X;\hat{\Lambda}) \to H^*_{\mu_p}(X;\hat{\Lambda}) \cong H^*(X;\hat{\Lambda}) [\![t, \theta]\!]$ is the $\hat{\Lambda}$-linear map acting on $x_0 \in H^*(X;k)$ as
    \begin{equation}
        \Sigma_x (x_0) = \sum_{\alpha \in H_2^{\mathrm{eff}}(X;\mathbb{Z})} \sum_{x_\infty \in B} \langle x_0 ; x_{eq}; x_\infty \rangle_{0, \alpha}  \ x_\infty^\vee \ z^\alpha \in H^*_{\mu_p}(X;\hat{\Lambda}),
    \end{equation}
    where $x_\infty^\vee$ denotes the Poincar\'e dual of the basis element $x_\infty \in B$.
\end{defn}
By extending $t, \theta$-linearly, this can be made into an endomorphism of $H^*(X;\hat{\Lambda}) [\![t, \theta]\!]$.

The actual construction, outlined in \cite[Section 4]{seidel-wilkins} or \cite[Section 2.4]{Lee23a}, has to go through symplectic enumerative geometry and moduli space of pseudo-holomorphic maps due to the lack of theory providing the required equivariant VFC $[\overline{\mathcal{M}}_\alpha]^{vir}_{\mu_p}$ to define the correlators $\langle x_0; x_{eq}; x_\infty \rangle_{0, \alpha}$. 

In the presence of a $T$-action on $X$, the equivariant generalizations of these operators are given in \cite[Definition 2.17]{Lee23b}. The operators are indexed by $b \in H^*_T(X;k)$, and are defined as $\hat{\Lambda}[\![t,\theta]\!]$-linear endomorphisms $\Sigma_b^T \in \mathrm{End} \left(H^*_{T, loc}(X;\hat{\Lambda})[\![t,\theta]\!]\right)$ of localized equivariant cohomology. The main technical result of \cite{Lee23b} can be summarized as follows:

\begin{thm}[{\cite[Corollary 3.17]{Lee23b}}]
    Fix a conical Hamiltonian symplectic resolution $X$ with its $T \times \mathbb{G}_m$-action with isolated $T$-fixed points. Fix a $T$-equivariant cohomology class $x \in H^*_T(X;k)$. Then the corresponding \emph{$T \times \mathbb{G}_m$-equivariant} quantum Steenrod operator
    \begin{equation}
        \Sigma_x^T: H^*_{T \times \mathbb{G}_m, loc} (X; \hat{\Lambda}) [\![{t, \theta}]\!] \to H^*_{T \times \mathbb{G}_m, loc} (X; \hat{\Lambda}) [\![{t, \theta}]\!],
    \end{equation}
    where $H^*_{T \times \mathbb{G}_m, loc} (X; \hat{\Lambda}) [\![{t, \theta}]\!] := H^*_{T \times \mathbb{G}_m} (X; \hat{\Lambda}) [\![{t, \theta}]\!] \otimes \mathrm{Frac}(H^*_{T \times \mathbb{G}_m}(\mathrm{pt}) )$ is the localized equivariant cohomology, is well-defined. 
\end{thm}

\begin{rem}
    It is an interesting open question to define the quantum Steenrod operators algebro-geometrically. On the other hand, power operations similar to the Steenrod operations exist for other cohomology theories, including complex $K$-theory, over which one can define the \emph{Adams operation}. Its quantum counterpart can be set up using algebraic geometry (cf. \cite{givental-lee, lee-K}), which will be studied in forthcoming work.
\end{rem}

\subsubsection{Properties of quantum Steenrod operations}\label{sssec:qst-properties}

In this subsection, we record the important properties of the ($T\times \mathbb{G}_m$-equivariant) quantum Steenrod operators, in particular its equivalence with the $p$-curvature of the quantum $D$-module. For our proofs, we only need these algebraic properties rather than their actual constructions.

Consider the localized equivariant cohomology with Novikov coefficients, $H^*_{T \times \mathbb{G}_m, loc} (X; \hat{\Lambda}) [\![{t, \theta}]\!] := H^*_{T \times \mathbb{G}_m} (X; \hat{\Lambda}) [\![{t, \theta}]\!] \otimes \mathrm{Frac}(H^*_{T \times \mathbb{G}_m}(\mathrm{pt}) )$. The quantum Steenrod operator $\Sigma_x^T$ is a $\hat{\Lambda} \otimes k[\![t, \theta]\!] \otimes H^*_{T \times \mathbb{G}_m}(\mathrm{pt})$-linear endomorphism of this vector space. Note that this group also carries a ($T \times \mathbb{G}_m$-equivariant) quantum product, which we denote by $\star_{T}$.

\begin{rem}
    The choice of the subscript $T$, instead of $T \times \mathbb{G}_m$, is somewhat unconventional, but by the deformation argument described in e.g. \cite[Section 1.1]{BMO11} the quantum product reduces to the classical product without working $\mathbb{G}_m$-equivariantly.
\end{rem}

The quantum Steenrod operators are related to the usual Steenrod operators as follows. Recall that given a cohomology class $x \in H^*_T(X;k)$, one has the (non-linear) map
\begin{align}
    H^*_T(X;k) &\to H^*_{T^p \rtimes \mu_p} (X^p;k) \\
    x &\mapsto x_{eq} := x \otimes x \otimes \cdots \otimes x.
\end{align}
Note that there is an action of $T \times \mu_p$ on $X$, where $\mu_p$ acts trivially. The diagonal inclusion $\Delta: X \to X^p$, which is $(T \times \mu_p, T^p \rtimes \mu_p)$-equivariant, induces a pullback map in equivariant cohomology. Through this pullback we define
\begin{defn}
    The \emph{total Steenrod power} of $x \in H^*_T(X;k)$ is the composition
    \begin{equation}
        \mathrm{St}^T (x) := \Delta^* (x_{eq}) \in H^*_{T \times \mu_p}(X;k) \cong H^*_T(X;k)[\![t, \theta]\!]
    \end{equation}
    where the last isomorphism is the K\"unneth isomorphism for the trivial $\mu_p$-action on $X$.
\end{defn}
The coefficients in $H^*_T(X;k)$ of $\mathrm{St}^T(x)$ as a polynomial in $t, \theta$ are (up to a sign, see \cite[Remark 1.3]{seidel-wilkins} or \cite[Section 2.1]{lonergan}) the usual Steenrod operations on $x$ in singular cohomology with positive characteristic coefficients, hence the name. 

\begin{prop}\label{prop:qst-properties}
    For $x \in H^*_T(X;k)$, the operator $\Sigma_x^T$ satisfies the following properties.
    \begin{enumerate}[topsep=-2pt, label=(\roman*)]
        \item $\Sigma_1^T = \mathrm{id}$ for $1 \in H^0_T(X;k)$, the unit class.
        \item $\Sigma_{c x}^T = c^p \Sigma^T_{x}$ for $c \in \Lambda' = k[z^\alpha: \alpha \in H_2^{\mathrm{eff}}(X;\mathbb{Z})]$ a constant.
        \item $\Sigma_x^T(x_0)|_{(t,\theta) = 0} = \overbrace{x \star_T \cdots \star_T x}^{p} \star_T \ x_0$.
        \item $\Sigma_x^T(x_0)|_{z^\alpha = 0} = \mathrm{St}^{T}(x) \smile_{T} x_0$, where $\mathrm{St}^T(b) \in H^*_{T \times \mathbb{G}_m}(X;k)\zpeq$ is the total Steenrod power in $T \times \mathbb{G}_m$-equivariant cohomology, and $\smile_T$ is the $(T\times \mathbb{G}_m)$-equivariant cup product.
        \item $\Sigma_{x \star_T x'}^T = (-1)^{\frac{p(p-1)}{2} |x| |x'|} \Sigma_x^T \circ \Sigma_{x'}^T$.
    \end{enumerate}
\end{prop}
Property (ii) is just Frobenius-linearity. Property (iii) claims that quantum Steenrod operators are \emph{equivariant} deformations of the $p$-th power operation in quantum cohomology, and (iv) claims that the quantum Steenrod operators are \emph{quantum} deformations of the Steenrod operations in mod $p$ cohomology. Property (v) is called the \emph{quantum Cartan relation}, which shows that quantum Steenrod operators $\Sigma_x^T$ collectively define an \emph{algebra action} from the (Frobenius-twisted) mod $p$ ($T$-equivariant) quantum cohomology algebra. In particular, the actions of $\Sigma_x^T$ and $\Sigma_{x'}^T$ graded commute for any $x, x'$ by the graded commutativity of $\star_T$.

The main result of \cite{Lee23b} is the identification of quantum Steenrod operators on degree $2$ classes for conical Hamiltonian symplectic resolutions with the \emph{$p$-curvature} of the quantum connection. Recall from \cref{sssec:eq-kah-dmodules} that the quantum connection is a $k[\![t, \theta]\!] \otimes H^*_{T \times \mathbb{G}_m}(\mathrm{pt})$-linear operator
\begin{equation}
    \nabla_x := t \partial_{\bar{x}} + x \ \star_T : H^*_{T \times \mathbb{G}_m, loc} (X; \hat{\Lambda}) [\![{t, \theta}]\!] \to H^*_{T \times \mathbb{G}_m, loc} (X; \hat{\Lambda}) [\![{t, \theta}]\!]
\end{equation}
where $\partial_{\bar{x}}$ acts on $z^\alpha \in \hat{\Lambda}$ by $\partial_{\bar{x}} z^\alpha = \langle \alpha, \bar{x} \rangle z^\alpha$. Here, $t$ is the equivariant parameter for the $\mu_p(\mathbb{C})$-action on the \emph{source} curve for holomorphic maps $u : \mathbb{P}^1 \to X$ counted for defining the quantum connection, which we later identify with the equivariant parameter for the $\mathbb{G}_m(\mathbb{C})$-action on the \emph{target} $X$ (see \cref{rem:hbar-vs-t}).

\begin{rem}
    The $\mu_p$-action on the domain $\mathbb{P}^1$ can be viewed as a discretized loop rotation action, in the sense that we restrict the $\mathbb{G}_m(\mathbb{C})$-action on $\mathbb{P}^1$ which fixes $z_0$ and $z_\infty$ to the action of its subgroup $\mu_p \subset \mathbb{G}_m(\mathbb{C})$ so that it preserves the set of marked points. This perspective is also useful for understanding the Frobenius center of quantized Coulomb branches.
\end{rem}

\begin{thm}[{\cite[Theorem 1.2]{Lee23b}}]\label{thm:qst-is-pcurv}
    Fix a conical Hamiltonian symplectic resolution $X$ with isolated $T$-fixed points. Then for almost all $p$ and $x \in H^2_T(X;\mathbb{F}_p)$, we have
    \begin{equation}
        \Sigma_x^T = \nabla_x^p - t^{p-1} \nabla_x + N
    \end{equation}
    where $N$ is a nilpotent operator.
\end{thm}
The right hand side expression, $\nabla_x^p - t^{p-1}\nabla_x$, is the operator known as the \emph{$p$-curvature} of the quantum connection. Although the connection $\nabla_x$ \emph{is not} $\hat{\Lambda}$-linear, the $p$-curvature $\nabla_x^p - t^{p-1} \nabla_x$ remarkably \emph{is} $\hat{\Lambda}$-linear.

Conjecturally the nilpotent error always vanishes, that is, the quantum Steenrod operators (on degree $2$ classes) are expected to be always equal to the $p$-curvature of the quantum connection. Current proof for this conjecture assumes generic semisimplicity, as follows. Recall that a collection of commuting linear operators $\{F_i\}$ (with $[F_i, F_j] = 0$) acting on a vector space $V$ over a field $K$ have \emph{jointly simple spectrum} if the associated simultaneous eigendecomposition of $V \otimes_K \overline{K}$ is into $1$-dimensional subspaces.

\begin{cor}[{\cite[Corollary 5.2]{Lee23b}}]\label{thm:qst-is-pcurv-onthenose}
    Suppose the quantum multiplication operators $y  \ \star_T$ associated to degree $2$ classes $y \in H^2_T(X;k)$ have jointly simple spectrum. Then for $x \in H^2_{T}(X;\mathbb{F}_p)$,
        \begin{equation}
        \Sigma_x^T = \nabla_x^p - t^{p-1} \nabla_x.
    \end{equation}
\end{cor}
\begin{proof}
Since $\Sigma_x^T - F_x = N$ is nilpotent by \cref{thm:qst-is-pcurv}, it suffices to show that $\Sigma_x^T$ and $F_x$ are simultaneously diagonalizable. 

    By \cref{thm:seidel-wilkins} below and the well-known flatness of $\nabla_x$, both $\Sigma_x^T$ and $F_x:= \nabla_x^p - t^{p-1} \nabla_x$ commute with each other, and also with any operator of the form $F_y := \nabla_y^p - t^{p-1} \nabla_y$ where $\nabla_y = t \partial_{\bar{y}} + y \ \star_T$ is the quantum connection for $y \in H^2_T(X;k)$.

    Hence, it suffices that the collection of operators $\{F_y\}$ for $y \in H^2_T(X;k)$ have jointly simple spectrum: since both $\Sigma_x^T$ and $F_x$ commute with $\{F_y\}$, they must be simultaneously diagonalizable in the corresponding eigenbasis of $\{F_y\}$.

    Now note that $F_y|_{t = 0} = (y \ \star_T)^p$, and therefore $\{y \ \star_T \}$ having jointly simple spectrum implies the same for $\{F_y\}$. This is our assumption, which gives the desired result.
\end{proof}

The following theorem establishes the covariant constancy of quantum Steenrod operators, which is also used in the proof of \cref{thm:qst-is-pcurv}.

\begin{thm}[{\cite[Theorem 1.4]{seidel-wilkins}}]\label{thm:seidel-wilkins}
    The quantum Steenrod operator $\Sigma_x^T$ commutes with $\nabla_{x'}$.
\end{thm}
\begin{cor}\label{cor:qst-acts-on-Dmod}
    Denote the $\mathbb{G}_m$-equivariant parameter generating $H^2_{\mathbb{G}_m}(\mathrm{pt})$ by $\hbar$. Then the specialization
    \begin{equation}
        \Sigma_x^T|_{\hbar = t} : H^*_{T \times \mathbb{G}_m, loc} (X; \hat{\Lambda}) [\![{\hbar, \theta}]\!] \to H^*_{T \times \mathbb{G}_m, loc} (X; \hat{\Lambda}) [\![{\hbar, \theta}]\!].
    \end{equation}
    where $t \in H^2_{\mu_p}(\mathrm{pt})$ is identified with $\hbar \in H^2_{\mathbb{G}_m}(\mathrm{pt})$, acts on the (specialized) quantum $D$-module $M_{Kah, reg}$ after localizing equivariant parameters.
\end{cor}
\begin{proof}
    In the definition of the quantum $D$-module (cf. \cref{defn:Dmod-quantum}), note that the defining action of $x' \in H^2_{T \times \mathbb{G}_m}(X)$ is exactly given by the specialization $\nabla_{x'}|_{\hbar = t}$. By \cref{thm:seidel-wilkins}, the specialization $\Sigma_{x}^T|_{\hbar = t}$ commutes with $\nabla_{x'}|_{\hbar = t}$, hence acts on the quantum $D$-module. Note that terms involving $\theta$ do not appear for $p \gg 0$, as rationally the cohomology groups of conical Hamiltonian symplectic resolutions are always concentrated in even degrees.
\end{proof}

\subsection{Frobenius-constant quantizations}\label{ssec:frobquant}
In characteristic $p$, the notion of a quantization as a flat $k[\![\hbar]\!]$-deformation still makes sense, but a striking new feature is that the resulting quantization may admit a large center. This observation motivates the definition of a Frobenius-constant quantization, the definition of which is due to \cite{bezrukavnikov-kaledin-quantp}. Here, we observe that Frobenius-constant quantizations give rise to an action of the (Frobenius-twisted) coordinate ring of the symplectic resolution on the $D$-module of twisted traces. Throughout our discussion, for any vector space, algebra, etc. $M$ over a field $k$ of characteristic $p$, we denote by 
\begin{equation*}
    M^{(1)} := M \otimes_k k
\end{equation*}
the Frobenius pullback with respect to the Frobenius map $\mathrm{Fr}: k \to k$ which maps $x$ to $x^p$.

\subsubsection{Frobenius-constant quantization operators}\label{sssec:frobquant-opers}
    First we recall the following definition:
    \begin{defn}[{\cite[Definition 1.4]{bezrukavnikov-kaledin-quantp}}]\label{defn:frob-const-quant}
        A \emph{Frobenius-constant quantization} of a Poisson variety $X$ over a field $k$ characteristic $p > 0$ is a quantization $A$ of $\mathcal{O}(X)$ and an algebra map
        \begin{equation}
            \Lambda : \mathcal{O}(X)^{(1)} \to Z (A)
        \end{equation}
        such that $\Lambda(f) \equiv f^p$ mod $\hbar^{p-1}$.
    \end{defn}
    Let us refer to the map $\Lambda$ in the data of a Frobenius-constant quantization as a \emph{quantized Frobenius}, as it lifts the usual, commutative Frobenius map $\mathcal{O}(X)^{(1)} \to \mathcal{O}(X)$ in characteristic $p$ to the quantization $A$.

    Now we fix a conical Hamiltonian symplectic resolution $X \to Y$, its universal deformation $\mathcal{X} \to \mathcal{Y}$ and the canonical quantization $\mathcal{A}$ of $\mathcal{O}(\mathcal{X})$. We assume that this quantization is equipped with the structure of a Frobenius-constant quantization, that is there exists a quantized Frobenius
    \begin{equation}
\Lambda : \mathcal{O}(\mathcal{X})^{(1)}_{\lambda} \to Z (\mathcal{A}_{p\lambda})
    \end{equation}
    where $\mathcal{O}(\mathcal{X})_\lambda$ denotes the $T$-weight $\lambda$ functions on the universal deformation and $\mathcal{A}_{p\lambda}$ is the $T$-weight $p \cdot \lambda$ part of the quantization. Such maps indeed exist for:

    \begin{itemize}
        \item Springer resolutions: classically the Frobenius center of asymptotic universal enveloping algebra (crystalline differential operators) gives the desired quantized Frobenius \cite[Proposition 3.5]{bezrukavnikov-kaledin-mckay} \cite[Section 1.2]{bezrukavnikov-mirkovic-rumynin},
        \item Coulomb branches for a gauge theory associated with $(G,V)$ (for abelian $G$, these recover hypertoric varieties), by the construction of \cite{lonergan}.
    \end{itemize}

    We use the map $\Lambda$ to define multiplication operators acting on the $D$-module of twisted traces as follows. Fix $\bar{x} \in \mathcal{O}(\mathcal{X})_0$, an element of the universal coordinate ring of torus weight $0$.

    \begin{prop}\label{prop:frob-acts-on-Dmod}
        Left multiplication by $\Lambda ( \bar{x})$ defines a well-defined, covariantly constant endomorphism of the $D$-module of twisted traces.
    \end{prop}
\begin{proof}
    To show that this is a well-defined operation, we must check that the action factors through the relations
    \begin{equation}
         J := \{ ab - z^\lambda ba : a \in \mathcal{A}_\lambda, b \in \mathcal{A}_{-\lambda} \}
    \end{equation}
    in the $D$-module of twisted traces. Indeed, by centrality, it is immediate that for a fixed $x \in \mathcal{O}(\mathcal{X})_0$ we have
    \begin{equation}
        \Lambda(\bar{x})(ab - z^\lambda ba) = a (b\Lambda(\bar{x}))-z^\lambda (b\Lambda(\bar{x}))a \in J.
    \end{equation}
    The commutation relation with the action of the differential operators (covariant constancy) can be checked similarly. Given $y \in \mathcal{A}^2_0$, we have
    \begin{align}
    y \cdot \Lambda(\bar{x}) (z^\mu \otimes a) &= y \cdot (z^{\mu} \otimes \Lambda(\bar{x})a) 
    \\ &= z^{\mu} \otimes (\hbar \langle \mu, \bar{y} \rangle \Lambda(\bar{x}) a + y \Lambda(\bar{x}) a) \\
    &= z^{\mu} \otimes \Lambda(\bar{x})\left( \hbar \langle \mu, \bar{y} \rangle a + ya \right) \\
    &= \Lambda(\bar{x}) \big(y \cdot (z^\mu \otimes a) \big),
    \end{align}
    where the second equality uses the centrality of $\Lambda(\bar{x})$.
\end{proof}

Since the quantized Frobenius $\Lambda$ is by definition an algebra map, the collection of multiplication operators $\Lambda({\bar{x}})$ define an algebra action of $\mathcal{O}(\mathcal{X})_0$, i.e.
\begin{equation}
    \Lambda(\bar{x} \bar{y}) = \Lambda(\bar{x}) \circ \Lambda(\bar{y}).
\end{equation}
This is mirror to the quantum Cartan relation from \cref{prop:qst-properties} (v).

We can extend the action to $k[z^\mu : \mu \in \mathbb{Z}_{\ge 0} \Sigma_{eq, +} ]  \otimes \mathcal{O}(\mathcal{X})_0$ by letting $z^\mu$ act by left multiplication by $z^{p\mu} \otimes 1$. For this extended action, the following is an immediate consequence of the definition of $D$-module of twisted traces:

\begin{prop}\label{prop:frob-acts-on-Dmod-factors-through-B}
    The action by left multiplication of images of  $\Lambda$ on the $D$-module of twisted traces factors through
        \begin{equation}
            \mathscr{B} (\mathcal{O}(\mathcal{X})) := k[z^\mu : \mu \in \mathbb{Z}_{\ge 0} \Sigma_{eq, +} ] \otimes \mathcal{O}(\mathcal{X})_0 / \{ 1 \otimes \bar{a}\bar{b} - z^\lambda \otimes \bar{a}\bar{b} : \bar{a} \in \mathcal{O}(\mathcal{X})_\lambda, \bar{b} \in \mathcal{O}(\mathcal{X})_{-\lambda } \}.
        \end{equation}
\end{prop}
\begin{proof}
    The elements in the relation $1 \otimes xy - z^\lambda \otimes yx$ act by $\Lambda(x)\Lambda(y) - z^{p \lambda} \Lambda(y) \Lambda(x)$; this is zero in the $D$-module of twisted traces, since $\Lambda(x) \in \mathcal{A}_{p\lambda}$ and $\Lambda(y) \in \mathcal{A}_{-p\lambda}$.
\end{proof}

\begin{rem}
    We consider the algebra $\mathscr{B}(\mathcal{O}(\mathcal{X}))$ as the mirror to the equivariant quantum cohomology algebra $H^*_{T^!}(X^! ; \Lambda')$; the isomorphism is induced by passing to $\hbar = 0$ in the isomorphism $M_{eq, reg} \cong M^!_{Kah, reg}$. Note that both left multiplication by images of $\Lambda$ and the quantum Steenrod operations $\Sigma_{x}^{T^!}$ define an \emph{algebra} action. In particular, the \emph{quantum} multiplication in quantum cohomology is mirror to \emph{ordinary} multiplication.
\end{rem}

\subsubsection{Lonergan's construction of quantized Frobenius for Coulomb branches}\label{sssec:frobquant-lonergan}
There is a general construction of the Frobenius-constant quantization of Coulomb branches for a gauge theory $(G,V)$ due to Lonergan, cf. \cite{lonergan}. We briefly sketch this construction, as (i) we later use the description to compute the quantized Frobenius for hypertoric varieties, and (ii) we would like to highlight the similarity of this construction and the construction of quantum Steenrod operations. For a complete description, the reader is referred to the original paper \cite{lonergan}.

The starting point is to consider the following generalization of the affine Grassmannian. Let $X$ be a curve, $S' \subseteq S$ a finite set and a subset. Denote by $\Delta_S(x)$ the union of formal disc neighborhoods of $x_s \in X$ for all $s \in S$. Denote by $\Delta_{S}^{S'}(x)$ the complement of the graphs of $\{ x_{s} \}_{s \in S'}$ in $\Delta_S(x)$, i.e. introducing punctures at points indexed by $s \in S'$.

\begin{defn}
    Let $G$ be a reductive group. The \emph{generalized Beilinson--Drinfeld Grassmannian} is the space $\mathrm{Gr}_{G | S, S'}$ given by isomorphism classes of triples $\{(x, \mathcal{P}, \phi)\}$ of $x \in X^S$, $\mathcal{P}$ a principal $G$-bundle over $\Delta_S(x)$, and $\phi$ a trivialization of $\mathcal{P}$ over $\Delta_S^{S'}(x)$.
\end{defn}
The construction of such a space as a fpqc-quotient of an inductive limit of ``homogeneous'' spaces (generalizing the description of the affine Grassmannian as a $G\brK$-homogeneous space) is given in \cite[Section 3.5]{lonergan}, where it is denoted $G_S^{S'}/G_S$.

\begin{rem}[{\cite[Remark 3.11]{lonergan}}]
    The case with $X = \mathbb{G}_a$, $|S| = 1$, $S= S' = \{0\}$ corresponds to the original affine Grassmannian. The case $S = S'$ is known as the (usual) Beilinson--Drinfeld Grassmannian $\mathrm{Gr}_{G, S}$ on $|S|$ points, whose fibers over $X^S$ are given by product of $|S|$ copies of $\mathrm{Gr}_G$.
\end{rem}

Similarly, we may introduce the moduli spaces $\mathcal{T}_{G, V | S, S'}$ with its projection $\pi: \mathcal{T}_{G,V|S,S'} \to \mathrm{Gr}_{G|S,S'}$ and the action map $ \mu_{S,S'}: \mathcal{T}_{G,V|S, S'} \to V_{S, S'}$. Here, $\mathcal{T}_{G, V | S, S'}$ parametrizes isomorphism classes of quadruples $(x, \mathcal{P}, \phi, v)$ where $x \in X^S$, $\mathcal{P}$ is a principal $G$-bundle over $\Delta_S(x)$, $\phi$ is a trivialization of $\mathcal{P}$ over $\Delta_S^{S'}(x)$, and $v$ is a section of the associated bundle $\mathcal{P} \times_{G} V$. Then, $\pi$ is the projection map defined by forgetting the section. The space $V_{S, S'}$ can be thought of as the product bundle over $\Delta_S^{S'}(x)$, which admits an action by the group $G_S^{S'}$. The action map $\mu_{S, S'}$ is then defined analogously to $\mu: \mathcal{T}_{G, V}  \to V(\!(z)\!)$, using the action by $G_{S}^{S'}$ on sections.

Correspondingly, $\mathcal{R}_{G, V|S, S'}$ can be described by taking the fiber product of $\mu_{S,S'}$ with $V_S \to V_{S, S'}$, where $V_S$ can be interpreted as the sections of the product bundle which have no poles. We refer the reader to \cite[\S 3.5, 3.6]{lonergan} for the accurate constructions.

As a special case, the Coulomb branches are constructed as convolution algebras of $\mathcal{R}_{\{1\},\{1\}}$ with $G_{\{1\}} = G\brO$-action. The idea due to Beilinson--Drinfeld, which is employed in Lonergan's construction of the Frobenius-constant quantization, is that this convolution algebra structure can be expressed as a composition of a (manifestly commutative) multiplication map and a specialization map (cf. \cite[Remark 3.27]{lonergan}).

Lonergan's construction of the quantized Frobenius uses the generalized Beilinson--Drinfeld Grassmannian corresponding to $X = \mathbb{G}_a = \mathrm{Spec} \ \mathbb{C}[t]$. Note that there is a projection map $\pi: X \to Y = X \gitq_0 \mu_p = \mathrm{Spec} \ \mathbb{C} [t^p]$ which is a quotient by $\mu_p$ action on $X$ given by a character $\chi$. By considering the generalized Beilinson--Drinfeld Grassmannian with $p$ points supported on the $\mu_p$-equivariant embedding $X \to X^p$ given by
$$
t \mapsto (\zeta \cdot t, \dots, \zeta^{p-1} \cdot t)
$$
where $\zeta$ denotes the generator of the action corresponding to $\chi$ and $X^p$ is equipped with the cyclic permutation action, and taking $\mu_p$-quotients, we obtain the following moduli space fibering over $Y = X \gitq_0 \mu_p$. For a $\mathbb{C}$-point $y$ of $Y$, denote by
\begin{equation}
    \pi^* \Delta_{S}^{S'}(y) := \mathrm{Spec} \left( \mathbb{C}[t] \left[\!\left[ \prod_{s \in S} (t^p - y_s) \right]\!\right]\left[ \prod_{s \in S'} (t^p - y_s)^{-1} \right] \right).
\end{equation}

\begin{defn}
    The \emph{moduli of triples in symmetric configuration} is $\mathcal{R}_{(p)}^{(p)}$ which parametrizes tuples $(y, \mathcal{P}, \phi, v)$ where $y \in Y$, $\mathcal{P}$ is a principal $G$-bundle over $\pi^* \Delta_1(y)$, $\phi$ is a trivialization over $\pi^* \Delta_1^1 (y)$, and $v$ is a section of the associated bundle with fibers $V\brO$ (i.e. section such that $\phi(v)$ extends over $\pi^*\Delta_1 (y)$).
\end{defn}

There is a natural action of $G_{(p)} = \{ (y, g) : y \in Y, g : \pi^* \Delta_1(y) \to G \}$ acting on $\mathcal{R}_{(p)}^{(p)}$ where we multiply the trivialization. (This is the analog of usual $G\brO$-action on $\mathcal{R}$). This can be extended to an action of $G_{(p)} \rtimes \mathbb{C}^*$ where $\mathbb{C}^*$ is the usual loop rotation action.

The moduli space $\mathcal{R}_{(p)}^{(p)}$ naturally fibers over $Y$. Correspondingly, there is a diagram
\begin{equation}
    \begin{tikzcd}
        \mathcal{R} \rar \dar & \mathcal{R}_{(p)}^{(p)} \dar & \lar \mathcal{R}^p \dar \\ 0 \rar & Y = \mathbb{G}_a & \lar 1
    \end{tikzcd}
\end{equation}
where the $G_{(p)} \rtimes \mathbb{C}^*$-action restricts to the usual $G\brO \rtimes \mathbb{C}^*$-action on the zero fiber $\mathcal{R}$ and $G \brO^p \rtimes \mu_p$-action on the generic fiber $\mathcal{R}^p$ (where $\mu_p$ acts by cyclic permutation of factors).

Lonergan uses this geometry to construct the Frobenius-constant quantization. Namely, one gets
\begin{equation}
    \Lambda_{Coulomb} : \mathcal{O}(X) := H_*^{G \brO} (\mathcal{R}) \to H_*^{G \brO^p \rtimes \mu_p} (\mathcal{R}^p) \to H_*^{G_{(p)} \rtimes \mathbb{C}^*}((\mathcal{R}_{(p)}^{(p)})^* ) \to H_*^{G \brO \rtimes \mu_p} (\mathcal{R}) = A
\end{equation}
by the composition of the following maps.

The first map is the Steenrod power map, defined for Borel--Moore homology using Lonergan's reinterpretation of the Steenrod operations in equivariant derived constructible categories. The second map is the descent isomorphism, induced by identifying the nonzero fiber over $1 \in Y$ as the slice of the $\mathbb{C}^*$-action on $\mathcal{R}_{(p)}^{(p)}$ over $Y^* = Y - \{0\}$).  The last map is the specialization map in Borel--Moore homology. 

\begin{thm}[{\cite[Theorem 3.28]{lonergan}}]\label{thm:lonergan-quantization}
    For $p > 2$, there exists a Frobenius-constant quantization of BFN Coulomb branch,
    \begin{equation}
        \Lambda_{Coulomb}: \mathcal{O}(\mathcal{M}_C(G, V))^{(1)} \to Z(\mathcal{A}_{\hbar}(G,V)),
    \end{equation}
    so that $\Lambda_{Coulomb}$ is a central algebra map.
\end{thm}

\subsection{The mod $p$ Quantum Hikita conjecture}\label{ssec:hikita-quantum-p}

We are now ready to state the mod $p$ version of the quantum Hikita conjecture. As in the beginning of the section, fix $k$ an algebraically closed field of characteristic $p \gg 0$ (say $k = \bar{\mathbb{F}}_p$). 

Assume that there is a pair $X, X^!$ of symplectically dual conical Hamiltonian symplectic resolutions, equipped with actions of algebraic tori $T$ and $T^!$. 

Our interpretation of a symplectically dual pair $(X/k, X^!/\mathbb{C})$ is the following. We first assume that the pair $(X/\mathbb{C}, X^!/\mathbb{C})$ is symplectically dual in the usual sense (cf. \cref{rem:s-dual-is-empirical}). Then we assume there exists $R$, a finite localization of $\mathbb{Z}$, over which $X$ is actually defined such that $X/\mathbb{C}$ is pulled back from $R \to \mathbb{C}$. When our $X/k$ is obtained by pulling back $X/R$ under $R \to R/\mathfrak{m} \cong k$ where $\mathfrak{m}$ is a maximal ideal of  $R$, we say that $X/k$ is dual to $X^!/\mathbb{C}$. (In particular, the dual of $X^!$ is not unique, corresponding to the choice of coefficients for invariants of $X^!$).

This way, we assume that $X$ and the torus $T$ are defined over the ground field $k$, but assume that $X^!$ and $T^!$ are still both defined over $\mathbb{C}$.

We assume the existence of a preferred Frobenius-constant quantization structure for the canonical quantization $\mathcal{A}$ of $X$; see \cref{sssec:frobquant-opers}.

Consider the $D$-module of twisted traces $M_{eq, reg}$ for $X$, and the quantum $D$-module $M_{Kah, reg}^!$ for $X^!$ with \emph{$k$-linear coefficients}. In particular, under our assumptions, both $D$-modules are defined over $k$.

\begin{conj}\label{conj:hikita-quantum-p}
    Let $X/k$ and $X^!/\mathbb{C}$ be symplectically dual.
    Then there exist
\begin{itemize}
    \item a bijection of the roots $\Sigma_{eq,+} \cong \Sigma_{Kah, +}^!$, 
    \item a compatible isomorphism of the ring of differential operators $R_{eq} \cong R_{Kah}^!$,
    \item and a compatible isomorphism of the $D$-modules $M_{eq, reg} \cong M_{Kah, reg}^!$ over $k$, taking $1 \in \mathcal{A}_0^0$ to $1 \in H^0(X^!)$,
\end{itemize}
    extending the isomorphisms from Hikita--Nakajima conjecture. Moreover, for a fixed $a \in \mathscr{B}(\mathcal{O}(\mathcal{X})) \cong H^*_{T^!}(X^!;\Lambda')$, the isomorphism $M_{eq, reg} \cong M^!_{Kah, reg}$ intertwines the action of
    \begin{itemize}
        \item $\Lambda(a)$ on $M_{eq, reg}$ (\cref{prop:frob-acts-on-Dmod}, \cref{prop:frob-acts-on-Dmod-factors-through-B}) and
        \item $\Sigma_{a}^{T^!}|_{\hbar = t}$ on $M^!_{Kah, reg}$ (\cref{cor:qst-acts-on-Dmod}, \cref{prop:qst-properties}).
    \end{itemize}
\end{conj}

\begin{rem}
Quantum Cartan relation (cf. \cref{prop:qst-properties} (v)) shows that the \emph{quantum} multiplication is mirror to ordinary multiplication in the Coulomb branch (as $\Lambda$ is an algebra map). This is compatible with the $D$-module isomorphism: by taking the $\hbar \to 0$ limits, we see that quantum multiplication by degree $2$ classes are mirror to multiplication by classes in $\mathcal{A}_0^2$.
\end{rem}

\begin{rem}\label{rem:higgs-coulomb}
    In general, we expect the result to be true when $X$ is (a resolution of) a BFN Coulomb branch defined over $k$, and $X^!$ is a Higgs branch over $\mathbb{C}$. The Frobenius-constant quantization operators $\Lambda$ acting on $M_{eq, reg}$ should arise from Lonergan's construction as recalled above.
\end{rem}

\section{Quantum Hikita conjecture mod $p$: Proofs}

In this section, we verify the mod $p$ 3D mirror symmetry conjecture in the case of Springer resolutions and hypertoric varieties. We describe a restricted quantized algebra structure on the ring of differential operators, giving rise to the action of $p$-curvature on both the quantum $D$-module and $D$-module of twisted traces. Then we identify the action of Frobenius-linear operators for degree $2$ elements with the action of $p$-curvature. The general result follows when the algebras acting on the $D$-modules are generated in degree $2$.

Our description of the Frobenius-constant quantization operators, and the proof that they are equal to the $p$-curvature, uses the theory of restricted Lie algebras (can be found in standard textbooks such as \cite[Section 7.10]{jantzen-book}) and the closely related concept of restricted quantized algebras (due to \cite[Section 1]{bezrukavnikov-kaledin-quantp}). For more detailed discussions of these structures we refer to the references above.

\subsection{Restricted quantized algebra and $p$-curvature}
The action of the Frobenius-linear operators on the $D$-modules can be understood as the action from the restricted powers in the ring of differential operators.

\begin{defn}[{\cite[Definition 1.9, reformulated]{bezrukavnikov-kaledin-quantp}}]\label{defn:restricted-quantized-algebra}
    A \emph{restricted quantized algebra} $A$ over $k$ is a flat $k[\![\hbar]\!]$-algebra with a Poisson bracket $\{-,-\}$ and a power operation $(-)^{[p]}: A \to A$ preserving $\hbar$ such that

    \begin{itemize}
        \item $(A, \{ -, - \})$ together with $(-)^{[p]}$ is a restricted Lie algebra:
        \begin{enumerate}
            \item for $\forall a \in k$ and $x \in A$, we have $(ax)^{[p]} = a^p x^{[p]}$,
            \item $\{ x^{[p]}, y \} = (\mathrm{ad} \ x)^p (y)$ for $x , y \in A$, and 
            \item $(x+y)^{[p]} = x^{[p]} + y^{[p]} + L(x,y)$, where $L(x,y)$ is the Lie polynomial defined by 
            \begin{equation*}
                \hbar^{p-1}L(x,y) = (x+y)^p - x^p - y^p;
            \end{equation*}
        \end{enumerate}
    \item we have
    \begin{equation}
        \hbar \{x, y \} = [x,y]  = xy - yx, \quad \quad \forall x,y \in A,
    \end{equation}
    i.e., the quantization condition;
    \item $(xy)^{[p]} = x^p y^{[p]} + x^{[p]}y^p - \hbar^{p-1} x^{[p]}y^{[p]} + P(x,y)$, where $P(x,y)$ is the universal polynomial defined by
    \begin{equation}
        \hbar^{p-1} P(x,y) = (xy)^p - x^p y^p.
    \end{equation}
    \end{itemize}
\end{defn}
We introduce an additional flatness assumption to the definition in \cite{bezrukavnikov-kaledin-quantp} so that the following characterization holds. Note that by \cite[Equation (1.7)]{bezrukavnikov-kaledin-quantp}, we have $(\hbar x)^{[p]} = \hbar x^p$ given that $s$ defined below is an algebra map using the flatness assumption. By passing to the classical limit $\hbar = 0$, one obtains a notion of a restricted Poisson algebra \cite[Definition 1.8]{bezrukavnikov-kaledin-quantp}.

\begin{lem}\label{lem:BK-defn-reformulation}
    Let $A$ be a flat $k[\![\hbar]\!]$-algebra over $k$ that is a quantization of $A/\hbar$, with a power operation $(-)^{[p]} : A^{(1)} \to A$ extending the Frobenius on $k$ and preserving $\hbar$, together with a Poisson bracket defined by the quantization condition $\hbar\{x, y\} = [x,y]$. Then $(A, \{-,-\}, (-)^{[p]})$ satisfies the axioms of a restricted quantized algebra  if and only if the map (the \emph{Artin--Schreier map}) 
    \begin{equation}
        s := (-)^p - \hbar^{p-1} (-)^{[p]} : A^{(1)} \to A
    \end{equation}
    is a central algebra map.
\end{lem}

\begin{proof}
    First assume that $s(x) = x^p - \hbar^{p-1} x^{[p]}$ is a central algebra map. We check the axioms of the restricted Lie algebra.
        Let $x, y \in A$. Then by centrality, we have
    \begin{equation}
        \begin{aligned}
            0&=  s(x)y - ys(x) \\&= x^p y - \hbar^{p-1}x^{[p]}y - y x^p + \hbar^{p-1}y x^{[p]} \\
            &= \hbar \{x^p, y\} - \hbar^p \{x^{[p]}, y\}.
        \end{aligned}
    \end{equation}
    Observe that $\hbar \{x^p , y\} = [x^p, y] = (\hbar\cdot  \mathrm{ad}\ x)^p y$, since in characteristic $p$ we have
    \begin{equation}
        [x^p, y] = \overbrace{[x, [x, \cdots,[x,y]\cdots]]}^{p};
    \end{equation}
    hence we have $\hbar^p (\mathrm{ad}  \ x)^p y = \hbar^p \{x^{[p]} , y\}$. Since there is no $\hbar$-torsion due to the flatness assumption, we conclude that $(\mathrm{ad}  \ x)^p y = \{ x^{[p]}, y\}$.

    By linearity, we have
    \begin{equation}
        \begin{aligned}
            s(x +y) &= (x+y)^p - \hbar^{p-1} (x+y)^{[p]} \\ &=
            x^p - \hbar^{p-1} x^{[p]} + y^p - \hbar^{p-1} y^{[p]} \\
            &= s(x) + s(y).
        \end{aligned}
    \end{equation}
    It follows that
    \begin{equation}
        \hbar^{p-1} L(x,y) = (x+y)^p - x^p - y^p = \hbar^{p-1} \left( (x+y)^{[p]} - x^{[p]} - y^{[p]} \right),
    \end{equation}
    so from flatness, we have $(x+y)^{[p]} - x^{[p]} - y^{[p]} = L(x,y)$. Together with the result above we conclude that $(A,\{-,-\}) $ is a restricted Lie algebra. 
    
    Similarly, $s(xy) = s(x)s(y)$ implies the condition $(xy)^{[p]} = x^p y^{[p]} + x^{[p]}y^p - \hbar^{p-1} x^{[p]}y^{[p]} + P(x,y)$ for the restricted quantized algebra structure.

    Conversely, let $A$ be a restricted quantized algebra. Then
    \begin{equation}
        \begin{aligned}
            s(x)y - ys(x) &= x^p y - \hbar^{p-1}x^{[p]}y - y x^p + \hbar^{p-1}y x^{[p]} \\
            &= \hbar \{x^p, y\} - \hbar^p \{x^{[p]}, y\} \\
            &= \{(\hbar x)^{[p]}, y\} - \hbar^p (\mathrm{ad} \ x)^p (y) \\
            &= (\mathrm{ad} \ \hbar x)^p (y) - \hbar^p (\mathrm{ad} \ x)^p (y) = 0
        \end{aligned}
    \end{equation}
    shows that indeed $s$ is central. The linearity and the algebra property of $s$ is checked exactly as above.
\end{proof}



\begin{assm}
    The symplectic resolution $X/k$, and its universal deformations and quantizations, are defined over $\mathbb{F}_p$ and $X/k$ is the pullback of $X/\mathbb{F}_p$ by the structure map $\mathbb{F}_p \to k$.
\end{assm}

 In particular, we assume that the canonical quantization $\mathcal{A} \cong \mathcal{A}_{\mathbb{F}_p} \otimes_{\mathbb{F}_p} k$ so that the projection $(\mathcal{A}_0^2)_{\mathbb{F}_p} \to \mathfrak{t}_{\mathbb{F}_p}$ in the quantization exact sequence is compatible with the $\mathbb{F}_p$-forms.

\begin{prop}\label{prop:restricted-quantized-structure}
    Over a field $k$ of characteristic $p$, the equivariant ring of differential operators
        \begin{equation}
        R_{eq} : =  k [z^\mu] \otimes \mathrm{Sym} \  \mathcal{A}^2_0,
    \end{equation}
    can be equipped with a structure of a restricted quantized algebra, with the power map defined by $a^{[p]} = a$ for $a \in (\mathcal{A}_0^2)_{\mathbb{F}_p}$, $\hbar^{[p]} = \hbar$, and $(z^\mu)^{[p]} = 0$. Similarly, the K\"ahler ring of differential operators
    \begin{equation}
        R_{Kah} := k [ z^\alpha] \otimes \mathrm{Sym} H^2_{T \times \mathbb{G}_m} (X)
    \end{equation}
    is equipped with a structure of a restricted quantized algebra such that $x^{[p]} = x$ for $x \in H^2_T(X;\mathbb{F}_p)$, $\hbar^{[p]} = \hbar \in H^2_{\mathbb{G}_m}(\mathrm{pt})$, and $(z^\alpha)^{[p]} = 0$.
\end{prop}
\begin{proof}
    We consider the case of $R_{eq}$ (the proof for $R_{Kah}$ is exactly the same). 

    We define $s$ on the generators of the algebra $R_{eq} := k[z^\mu] \otimes \mathrm{Sym} \ \mathcal{A}_0^2$, and extend by the property that $s$ is a $k$-linear algebra map. The algebra property determines the action of $(-)^{[p]}$, since it mandates that
    \begin{equation}
        \hbar^{p-1}(ab)^{[p]} = \hbar^{p-1} \left( a^p b^{[p]} + a^{[p]}b^p - \hbar^{p-1} a^{[p]}b^{[p]} + P(a,b)\right).
    \end{equation}
    To see that this is well-defined, we consider the defining relation $ a z^\mu  = z^\mu (a + \hbar \langle \mu, \bar{a} \rangle )$ for $a \in (\mathcal{A}_0^2)_{\mathbb{F}_p}$. Then $s(az^\mu) = s(a) s(z^\mu) = (a^p - \hbar^{p-1}a) z^{p\mu}$ is indeed equal to $s(z^\mu) s(a + \hbar \langle \mu, \bar{a} \rangle) = z^{p\mu} (a^p - \hbar^{p-1} a)$, because
        \begin{equation}
        s(a) = a^p - \hbar^{p-1}a = \prod_{j \in \mathbb{F}_p} (a- j \hbar),
    \end{equation}
    is invariant under shift by integer multiples of $\hbar$ (using that $\bar{a} \in \mathfrak{t}_{\mathbb{F}_p}$), and $z^{p\mu}$ is central (see below).
    
    Next, we check the centrality of the image $s := (-)^p - \hbar^{p-1}(-)^{[p]}$ on the generators $a \in \mathcal{A}_0^2$ and $z^\mu$. By Frobenius linearity, it suffices to check for $a \in (\mathcal{A}_0^2)_{\mathbb{F}_p}$. The centrality follows from the defining commutation relations in $R_{eq}$. First we show that $s(a) = a^p - \hbar^{p-1} a$ for $a \in (\mathcal{A}_0^2)_{\mathbb{F}_p}$ commutes with $z^\mu$. To see this, again note that
    $s(a) = a^p - \hbar^{p-1}a = \prod_{j \in \mathbb{F}_p} (a- j \hbar)$,
    and that
    \begin{equation}
        [a - j\hbar, z^{\mu} ] = [a, z^\mu] - [j\hbar , z^\mu] = [a,z^\mu] = \hbar \langle \mu, \bar{a} \rangle z ^\mu
    \end{equation}
    from the defining relation of $R_{eq}$. It then follows that
    \begin{align}
        s(a) z^\mu &= \prod_{j \in \mathbb{F}_p} (a-j\hbar) \cdot  z^\mu  \\ &= z^\mu \prod_{j \in \mathbb{F}_p} \big( (a - j\hbar) + \hbar \langle \mu, \overline{a} \rangle \big)
        \\ &= z^\mu s (a + \hbar \langle \mu, \bar{a} \rangle)
        \\&= z^\mu  s(a),
    \end{align}
    
    as desired. Finally we show that $s(z^\mu) = z^{p\mu}$ commutes with $a \in (\mathcal{A}_0^2)_{\mathbb{F}_p}$:
    \begin{equation}
        a z^{p\mu} = z^{p\mu} a + \hbar \langle p\mu, \bar{a} \rangle z^{p\mu} = z^{p\mu} a. \qedhere
    \end{equation}
    \end{proof}

    The image of the Artin--Schreier map $s$ gives rise to central endomorphisms of the $D$-module:
    \begin{defn}
        The \emph{$p$-curvature} of the $D$-module $M_{eq, reg}$ ($M_{Kah, reg}$, resp.) are the endomorphisms of $D$-modules given by the action of $s(R_{eq}) \subseteq R_{eq, reg}$ ($s(R_{Kah}) \subseteq R_{Kah, reg}$, resp.) for the Artin--Schreier map $s : R^{(1)} \to R$ constructed in \cref{prop:restricted-quantized-structure}.
    \end{defn}

    The fact that the action of $s(R)$ is an endomorphism of $D$-modules (that is, \emph{covariant constancy} of the $p$-curvature) is an immediate consequence of the centrality of the image of $s$.

    In the context of \cref{conj:hikita-quantum-p}, the isomorphism of $D$-modules $M_{eq, reg} \cong M^!_{Kah, reg}$ for symplectically dual pairs $X/k$ and $X^!/\mathbb{C}$ compatible with $R_{eq} \cong R^!_{Kah}$ implies the identification of $p$-curvature operators for both $D$-modules. Hence, in the case where $\mathcal{A}_0$ and $H^*_{T^! \times \mathbb{G}_m}(X^!)$ are generated by (conical and cohomological, resp.) degree $2$ elements, \cref{conj:hikita-quantum-p} can be reduced to the following claims.

    \begin{conj}\label{conj:pcurv-for-trace}
        For the $D$-module $M_{eq, reg}$ of twisted traces, the $p$-curvature endomorphism given by $s(a)$ for $a \in \mathcal{A}_0^2$ can be identified with the action of the Frobenius-constant quantization operator $\Lambda(\bar{a})$ (see \cref{sssec:frobquant-opers} for the distinguished choice of $\Lambda$).
    \end{conj}

    \begin{conj}[{cf. \cref{conj:p-curvature}}]\label{conj:pcurv-for-quantum}
        For the (specialized) quantum $D$-module $M_{Kah, reg}$, the $p$-curvature endomorphism given by $s(x)$ for $x \in H^2_{T \times \mathbb{G}_m}(X)$ can be identified with the action of the (specialized) quantum Steenrod operator $\Sigma_x^T |_{\hbar = t}$.
    \end{conj}

The strategy for verifying the mod $p$ quantum Hikita conjecture (\cref{conj:hikita-quantum-p}), modulo the identification of $D$-modules, is to verify \cref{conj:pcurv-for-trace} and \cref{conj:pcurv-for-quantum} for 3D mirror pairs. Namely, the isomorphism of $D$-modules in characteristic $p$ in particular identifies the $p$-curvature of the two $D$-modules, which we interpret as the central endomorphisms we construct.

\subsection{Springer resolutions}
As in \cref{sssec:exmp-springer}, fix the Springer resolution $X = T^*G/B \to \mathcal{N}$, where $\mathcal{N} \subseteq \mathfrak{g}^*$ is the nilpotent cone. This is a conical Hamiltonian symplectic resolution, with the universal deformation $\mathcal{X}$ given by the Grothendieck--Springer simultaneous resolution; the deformation fits into the picture
\begin{center}
    \begin{tikzcd}
        & X = T^*G/B \rar \dar & \mathcal{X} = \widetilde{\mathfrak{g}}^* \dar \\ & \mathcal{N}  = \mathfrak{g}^* \times_{\mathfrak{t}^*/W} 0  \rar & \mathfrak{g}^* \times_{\mathfrak{t}^*/W} \mathfrak{t}^* 
    \end{tikzcd}
\end{center}
where the vertical maps are the affinizations. In particular, the affinization of $\mathcal{X}$ is given by $\mathfrak{g}^* \times_{\mathfrak{t}^*/W} \mathfrak{t}^*$ where we take the adjoint quotient (i.e. characteristic polynomial) map $\mathfrak{g}^* \to \mathfrak{t}^*/W$ and the Weyl quotient $\mathfrak{t}^* \to \mathfrak{t}^*/W$. 

To describe the canonical quantization, note that the universal enveloping algebra $U\mathfrak{g}$ carries the PBW filtration with $\mathfrak{g}$ in degree $2$. With respect to this filtration, one can form the corresponding Rees algebra (with Rees parameter $\hbar$), denoted $U_\hbar \mathfrak{g}$, with the commutator given by $[x, y] = \hbar \{ x, y\}$ for $x, y \in \mathfrak{g}$. 
The canonical quantization is then given by $\mathcal{A} = U_{\hbar} \mathfrak{g} \otimes_{Z_{\mathrm{HC}}(U_\hbar \mathfrak{g})} \mathrm{S} (\mathfrak{t})[\hbar]$ where $\mathrm{S}(\mathfrak{t}) := \mathcal{O}(\mathfrak{t}^*) $ is the symmetric algebra on $\mathfrak{t}$; here, the structure map
\begin{equation}
    Z_{\mathrm{HC}}(U_\hbar \mathfrak{g}) \cong\mathrm{S} (\mathfrak{t}) ^W [\hbar] \subseteq \mathrm{S} (\mathfrak{t})[\hbar]
\end{equation}
is the Harish--Chandra isomorphism.

We denote by $T$ the maximal torus in $G$, which acts in a Hamiltonian fashion on $T^*(G/B)$, and $T \times \mathbb{G}_m$ for the product of the Hamiltonian torus and the copy of $\mathbb{G}_m$ which scales the holomorphic symplectic form.

The cohomology $H^*_{T \times \mathbb{G}_m}(X)$ and the weight degree $0$ part of the quantization $\mathcal{A}_0$ are both generated by elements of (cohomological and conical, resp.) degree $2$, see \cite[Section 7.6, 7.1]{KMP21}. Hence it suffices to verify \cref{conj:pcurv-for-quantum}, \cref{conj:pcurv-for-trace}.

\subsubsection{The $D$-module of twisted traces}
The affinization of the universal deformation of $T^*G/B$ is given by
\begin{equation}
    \mathcal{O}(\mathcal{X}) \cong \mathcal{O}(\mathfrak{g}^*) \otimes_{\mathrm{S}(\mathfrak{t})^W } \mathrm{S}(\mathfrak{t}),
\end{equation}
where $\mathrm{S}(\mathfrak{t})^W \cong k [\mathfrak{t}^*]^W \cong k [\mathfrak{g}^*]^G \hookrightarrow \mathcal{O}(\mathfrak{g}^*)$ is the inclusion of $G$-invariant polynomials.

Correspondingly, the universal quantization of $T^*G/B$ is given by
\begin{equation}
    \mathcal{A} \cong U_\hbar \mathfrak{g} \otimes_{Z(U_\hbar \mathfrak{g})} \mathrm{S} ( \mathfrak{t})[\hbar];
\end{equation}
the structure map $\varphi: Z(U_\hbar \mathfrak{g}) \to \mathrm{S}( \mathfrak{t})^W [\hbar]$ is the Harish--Chandra isomorphism. 

Then the $0$-weight part of the quantization is given by
\begin{equation}
    \mathcal{A}_0 = (U_\hbar \mathfrak{g})_0 \otimes_{Z (U_\hbar \mathfrak{g})} \mathrm{S}(\mathfrak{t}) [\hbar]
\end{equation}
where the $Z(U_\hbar \mathfrak{g}) \to (U_\hbar \mathfrak{g})_0$ is the obvious inclusion.

The following is the key observation that can be extracted from \cite[Section 1.3.3]{bezrukavnikov-mirkovic-rumynin}:

\begin{prop}\label{prop:frob-quant-springer}
    There is a quantized Frobenius map
        \begin{equation}
    \Lambda : \mathcal{O}(\mathcal{X})^{(1)} \to Z(\mathcal{A}),
    \end{equation}
    given by sending
    \begin{equation}
        x \otimes f \in \mathcal{O}(\mathfrak{g}^*)^{(1)} \otimes_{\mathrm{S}(\mathfrak{t})^{W, (1)}} \mathrm{S}(\mathfrak{t})^{(1)} \mapsto (x^p - \hbar^{p-1} x^{[p]}) \otimes (f^p - \hbar^{p-1} f^{[p]})
    \end{equation}
    where $(-)^{[p]} : U_\hbar \mathfrak{g}^{(1)} \to U_\hbar \mathfrak{g}$ is the restricted $p$-th power operation from the restricted Lie algebra structure on $\mathfrak{g}$.
\end{prop}

We assume the following simple lemma.

\begin{lem}\label{lem:restricted-power-id-on-h}
    Let $\mathfrak{g}$ be a semisimple Lie algebra over $k$ with its restricted Lie algebra structure, in particular equipped with its natural $p$-th power operation $(-)^{[p]} : \mathfrak{g}^{(1)} \to \mathfrak{g}$. Then for $\sigma \in \mathfrak{t}_{\mathbb{F}_p} \subseteq \mathfrak{g}_{\mathbb{F}_p}$ an element of the Cartan subalgebra, we have $\sigma^{[p]} = \sigma$. 
\end{lem}
\begin{proof}
        Note that the adjoint representation of semisimple $\mathfrak{g}$ is faithful, and is a representation of restricted Lie algebras (in the sense that $x^{[p]}$ acts by $p$ iterates of  $x$, which follows from the definition of the restricted Lie algebra structure on $\mathfrak{g}$). Since $\sigma \in \mathfrak{t}_{\mathbb{F}_p}$ acts by a weight $\lambda \in \mathbb{F}_p$ on $\mathfrak{g}$, and $\lambda^p = \lambda$, it follows that $\sigma^{[p]} = \sigma$.
\end{proof}

\begin{proof}[Proof of \cref{prop:frob-quant-springer}]
    The fact that $(-)^p - \hbar^{p-1}(-)^{[p]} : \mathcal{O}(\mathfrak{g}^{*} )^{(1)} \to Z(U_\hbar \mathfrak{g})$ gives a central algebra map is proven in \cite[Section 1.3.3]{bezrukavnikov-mirkovic-rumynin}. It suffices to show that the extension of this map to $ \mathcal{O}(\mathcal{X})^{(1)} \cong \mathcal{O}(\mathfrak{g}^*)^{(1)} \otimes_{\mathrm{S}(\mathfrak{t})^{W, (1)}} \mathrm{S}(\mathfrak{t})^{(1)}$ is well-defined.

    For this, consider a $W$-invariant polynomial function $f$ on $\mathfrak{t}^*$, that is an element $f = \sum c_I \sigma^I \in \mathrm{S}(\mathfrak{t})^{W}$, where $\sigma^I = \sigma_1^{i_1} \cdots \sigma_r^{i_r}$ is a monomial in (some fixed) basis $\{\sigma_i\}$ of $\mathfrak{t}_{\mathbb{F}_p}$. By the algebra property of the morphism $(-)^p - \hbar^{p-1}(-)^{[p]}$, on one hand we see that
    \begin{equation}
        \Lambda(f \otimes 1) := \sum_I c_I^p \prod_{j=1}^r \left(\sigma_j^p - \hbar^{p-1} \sigma_j^{[p]}\right)^{i_j} \otimes 1 \in Z(U_\hbar \mathfrak{g}) \otimes \mathrm{S}(\mathfrak{t})[\hbar].
    \end{equation}
    On the other hand, for the presentation $f \otimes 1 = 1 \otimes f$ the definition implies
    \begin{equation}
        \Lambda(1 \otimes f) = 1 \otimes \sum_I c_I^p \prod_{j=1}^r  \left(\sigma_j^p - \hbar^{p-1} \sigma_j\right)^{i_j}    \end{equation}

    First, note that by \cref{lem:restricted-power-id-on-h} we have $\sigma_j^{[p]} = \sigma_j$ for all basis elements $\sigma_j \in \mathfrak{t}$.

    Next, we explain that passing from the first tensor factor in $Z(U_\hbar \mathfrak{g}) \otimes \mathrm{S}(\mathfrak{t})[\hbar]$ involves a shift. Consider $z \otimes 1 = 1 \otimes \varphi(z)$ where $\varphi : Z(U_\hbar \mathfrak{g}) \to \mathrm{S}(\mathfrak{t})[\hbar]$ is the Harish--Chandra morphism. Fix $\rho \in \mathfrak{t}^*$ the half sum of positive roots. The Harish--Chandra morphism satisfies that the image $\varphi(z)|_{\sigma \mapsto \sigma - \hbar \langle \rho, \sigma \rangle} \in \mathrm{S}(\mathfrak{t})[\hbar]$ considered as an element of $U_\hbar \mathfrak{g}$ (by the natural morphism $\mathrm{S}(\mathfrak{t})[\hbar] \to U_\hbar \mathfrak{g}$ coming from the inclusion $\mathfrak{t} \to \mathfrak{g}$) is equal to $z$.

    Now observe that
    \begin{equation}
        \left( \sigma^p - \hbar^{p-1} \sigma \right) = \prod_{k \in \mathbb{F}_p} ( \sigma -k \hbar)
    \end{equation}
    is invariant under any shifts $\sigma \mapsto \sigma - \hbar \langle \rho, \sigma \rangle$. Hence it follows that $\Lambda(1 \otimes f) = \Lambda ( f \otimes 1)$. 
\end{proof}

In particular, the computation of the map immediately verifies \cref{conj:pcurv-for-trace} for the Springer resolution:

\begin{prop}
    For $a \in \mathcal{A}_0^2 \cong \mathfrak{t} \oplus k \hbar \oplus \mathfrak{t} \subseteq R_{eq}$, the action of $\Lambda(\bar{a})$ on $M_{eq, reg}$ is equal to the action of $p$-curvature $s(a)$.
\end{prop}
\begin{proof}
    Indeed, from the construction of the quantized Frobenius and \cref{lem:restricted-power-id-on-h}, it immediately follows that the action of $\bar{a} \in \mathfrak{t} \oplus \mathfrak{t}$ is given by $a^p - \hbar^{p-1} a^{[p]} = s(a)$, as desired.
\end{proof}

\subsubsection{The quantum $D$-module}
The quantum $D$-module of the Springer resolution is computed in \cite[Theorem 1.2]{BMO11}, where the quantum connection is identified with the affine KZ connection (or the Dunkl connection). In characteristic $p$, the mod $p$ quantum $D$-module of the Springer resolution was studied in \cite{Lee23b}, and the following result was obtained as a special case of \cref{thm:qst-is-pcurv}. Denote by $\nabla_x = t \partial_{\bar{x}} + x \star_{T}$ the ($T \times \mathbb{G}_m$-equivariant) quantum connection for $ X = T^*(G/B)$.

\begin{prop}[{\cite[Theorem 1.2]{Lee23b}}]\label{thm:qst-is-pcurv-springer}
    Let $X = T^*(G/B)$ be the Springer resolution. Fix a degree $2$ class $x \in H^2_{T}(X;\mathbb{F}_p)$; then for almost all $p$, we have
    \begin{equation}
        \Sigma_x^T = \nabla_x^p - t^{p-1} \nabla_x,
    \end{equation}
    that is, the quantum Steenrod operation on degree $2$ classes agree with the $p$-curvature.
\end{prop}

\begin{proof}(cf. {\cite[Corollary 5.2]{Lee23b}})
    As in \cref{thm:qst-is-pcurv-onthenose}, it suffices to show that there is some $x \in H^2_{T}(X;k)$ such that $x \star$ has simple spectrum for generic $p$. Since $x \star$ is a formal deformation of the classical cup product operator $x \cup$, it suffices to check that $x \cup$ has simple spectrum for generic $p$. 
    
    For $x \in H^2_T(X;\mathbb{Z})$, there is a corresponding line bundle $L_x$ on $X = T^*(G/B)$ such that $c_1^T(L_x) = x$. Note that the fixed point basis $H^*_{T \times \mathbb{G}_m}(X^T)$ is an eigenbasis for the operator $x \cup$, and the eigenvalues are given by the weights of $L_x$ restricted to the isolated fixed points (indexed by the Weyl group $W$). The weight at $1 = B/B \in X^T$ is tautologically given by $x$, and the eigenvalues at other fixed points $wB/B \in X^T$ are given by the Weyl group orbits of $w \cdot x$. By choosing $x$ generically (e.g. in the dominant Weyl chamber), it can be achieved that $w \cdot x$ are all distinct for $w \in W$ for generic $p \gg 0$.
\end{proof}

This result verifies \cref{conj:pcurv-for-quantum} for the Springer resolution, hence concludes the proof of the mod $p$ quantum Hikita conjecture for the Springer resolution.

\subsection{Hypertoric varieties}
We begin by introducing two descriptions of hypertoric varieties as Higgs branches and Coulomb branches of gauge theories with abelian group $G$. In both descriptions, one can explicitly understand the structure of quantizations and the Frobenius-constant quantization operators: in the Higgs description, this is given by (crystalline) differential operators which is well-studied in \cite[Section 4.3]{stadnik}, and in the Coulomb description, this is due to \cite{lonergan} as reviewed in \cite[Theorem 3.1]{webster-coulombI}. We review the computation of Frobenius-constant quantizations for both descriptions. The Coulomb branch perspective is an instance of Remark \ref{rem:higgs-coulomb}.

\subsubsection{Definition: the Higgs description}
Hypertoric varieties are analogs of toric varieties in the quaternionic setting, and can be obtained as a hyperK\"ahler quotient of an abelian group action on an affine space in the \emph{Higgs description}.

Consider $\mathbb{A}^n$ together with the coordinate scaling action of the split algebraic torus $(\mathbb{G}_m)^n$. For a fixed map $(\mathbb{G}_m)^n \to T$ for $T$ some split algebraic torus, there is an associated exact sequence of algebraic groups
\begin{equation}
    \begin{tikzcd}
        1 \rar & K \rar & (\mathbb{G}_m)^n \rar & T \rar & 1.
    \end{tikzcd}
\end{equation}
By restriction, $\mathbb{A}^n$ admits a $K$-action, which extends to a Hamiltonian $K$-action on $T^*\mathbb{A}^n$ (with its standard algebraic symplectic form). Denote the corresponding moment map by $\mu : T^*\mathbb{A}^n \to \mathfrak{k}^*$, where $\mathfrak{k}^* := \mathrm{Hom}(K, \mathbb{G}_m)$ is the Lie coalgebra.
\begin{defn}
    The \emph{(Higgs) hypertoric variety} $ Y_H := \mathcal{M}_H(K, \mathbb{A}^n) =  \mu^{-1}(0) /\!\!/_{0} K$  is the hyperK\"ahler quotient for the trivial stability condition $0 \in \mathfrak{k}^*$.
\end{defn}

A resolution of $Y_H$ can be obtained by choosing a generic stability condition $\chi \in \mathfrak{k}^*$ and performing the associated projective GIT quotient $X_H : = \mu^{-1}(0)/\!\!/_\chi K$; note there is a natural map $X_H \to Y_H$.

The universal deformation $\mathcal{X}_H$ of $X_H$ is obtained by varying the moment map parameter, and its affinization is given by $\mathcal{Y}_H = T^*\mathbb{A}^n /\!\!/_0 K = \mathrm{Spec}\ \mathcal{O}(T^*\mathbb{A}^n)^K = \mathrm{Spec}\  k[x_i, y_i]^K$.

These fit into the following picture: 
\begin{center}
    \begin{tikzcd}
        & X_H = \mu^{-1}(0)/\!\!/_\chi K \rar \dar & \mathcal{X}_H  = T^* \mathbb{A}^n /\!\!/_\chi K \dar \\ & Y_H = \mu^{-1}(0)/\!\!/_0 K \rar  & \mathcal{Y}_H = T^*\mathbb{A}^n /\!\!/_0 K 
    \end{tikzcd}.
\end{center}

The canonical quantization in the Higgs description is obtained by taking the Weyl quantization of $\mathcal{O}(T^*\mathbb{A}^n)^K$. Namely, one considers the ring of (asymptotic) differential operators $\mathcal{D}(\mathbb{A}^n)$ on $\mathbb{A}^n$, generated by $x_i$ and $\partial_i$ for $i = 1, \dots, n$ with relation $[\partial_i, x_i] = \hbar$. The $(\mathbb{G}_m)^n$ action on $\mathbb{A}^n$ induces a grading by $\mathbb{Z}^n \cong X^\bullet(\mathbb{G}_m^n)$ on $\mathcal{D}(\mathbb{A}^n)$. Then the canonical quantization is obtained by taking the $K$-invariant part, namely $\mathcal{A} := \bigoplus_{\lambda \in \mathfrak{t}^*} \mathcal{D}(\mathbb{A}_n)_\lambda$.

\begin{lem}
The zeroth weight part of $\mathcal{A}$, that is $\mathcal{A}_0 = \mathcal{D}(\mathbb{A}^n)_0$, is generated by $x_i \partial_i$ for $i = 1, \dots, n$. 
\end{lem}
\begin{proof}
    Note that certainly $x_i \partial_i \in \mathcal{D}(\mathbb{A}^n)_0$, and consequently $\partial_i x_i = \hbar + x_i \partial_i \in \mathcal{D}(\mathbb{A}^n)_0$. A general monomial of weight grading zero is a product of $x_i$'s and $\partial_i$'s such that for every index $i$, the number of occurrences of $x_i$ and $\partial_i$ is equal. The fact that these are all generated by $x_i \partial_i$ follows from an easy induction involving the identity $x_i^k \partial_i^k = (x_i \partial_i) (x_i\partial_i + \hbar) \cdots (x_i\partial_i + (k-1)\hbar)$, see \cite[Sections 3.1, 3.2]{BLPW-hypertoric-category-O} for more details.
\end{proof}

In general, we have generators $m^\lambda := x^{\lambda_+} \partial^{\lambda_-}$ so that $\mathcal{D}(\mathbb{A}_n)_{\lambda} = \mathcal{D}(\mathbb{A}_n)_0 m^\lambda$ for $\lambda = \lambda_+ - \lambda_-$.

The cohomology $H^*_{T \times \mathbb{G}_m}(X_H)$ and the weight degree $0$ part of the quantization $\mathcal{A}_0$ are both generated by elements of (cohomological and conical, resp.) degree $2$, see \cite[Proposition 6.10, Section 6.2]{KMP21}. Hence it suffices to verify \cref{conj:pcurv-for-quantum}, \cref{conj:pcurv-for-trace}.

\subsubsection{Definition: the Coulomb description}
Alternatively, hypertoric varieties can also be constructed using BFN Coulomb branches. Again, let $V = \mathbb{A}^n$ be equipped with its coordinate-wise scaling action $ (\mathbb{G}_m)^n$, and consider it as a representation of the abelian group $G=K$ by the exact sequence
    \begin{equation}
    \begin{tikzcd}
        1 \rar & K \rar & (\mathbb{G}_m)^n \rar & T \rar & 1.
    \end{tikzcd}
\end{equation}

First, note that $\mathrm{Gr}_K := K(\!(z)\!)/K[\![z]\!]$ is just a copy of the cocharacter lattice $X_\bullet(K)$. We denote the generators by $z^\nu$ for $\nu \in X_\bullet(K)$. It is then easy to identify the BFN Steinberg variety and the moduli of triples for $(G, V) = (K, \mathbb{A}^n)$ as
\begin{equation}
    \mathcal{T}_{K, \mathbb{A}^n} := \bigsqcup_{\nu \in X_\bullet(K)} \{z ^\nu \} \times z^\nu V\brO, \quad \mathcal{R}_{K, \mathbb{A}^n} := \bigsqcup_{\nu \in X_\bullet(K)} \{z ^\nu \} \times (V\brO \cap z^\nu V\brO).
\end{equation}
For the description of $\mathcal{R}_{K, \mathbb{A}^n}$, we used the fact that the action map $\mu: \mathcal{T}_{G, V} \to V\brK$ simply takes $\{z^\nu\} \times z^\nu V\brO$ to $z^\nu V\brO$. We denote the fundamental classes (in the Borel--Moore homology) of components in $\mathcal{R}_{G,V}$ as $r_\nu$; these are also called the \emph{monopole operators}. The following statement is evident.

\begin{prop}
    The BFN Coulomb branch is additively $\mathcal{O}(Y_C) := H_*^{K\brO} (\mathcal{R}) = \bigoplus_{\nu \in \mathfrak{k}} H^*(BK) \cdot r_\nu$.
\end{prop}

The following is a computation of Braverman--Finkelberg--Nakajima of the ring structure of the Coulomb branch in the abelian case.

\begin{prop}[{\cite[4(i, vii)]{BFN2}}]\label{prop:abelian-coulomb-branch-is-hypertoric}
    The BFN Coulomb branch $\mathcal{M}_C(K, \mathbb{A}^n) = \mathrm{Spec} \ H_*^{K\brO}(\mathcal{R};k)$ can be identified as the GIT quotient $\mu_{T^\vee}^{-1}(0) /\!\!/_0 T^\vee$ for the trivial stability condition $0 \in \mathfrak{t} = \mathrm{Lie}(T^\vee)^*$.
\end{prop}

The Coulomb branch is therefore a hypertoric variety obtained by the hyperK\"ahler reduction construction on $T^*\mathbb{A}^n$ associated to the dual sequence
\begin{equation}
    \begin{tikzcd}
        1 \rar & T^\vee \rar & (\mathbb{G}_m)^n \rar & K^\vee \rar & 1.
    \end{tikzcd}
\end{equation}

Note that for $(G,V) = (K, \mathbb{A}^n)$, the generators of the Coulomb branch algebra are given by $\nu \in \mathfrak{k}$ (cocharacters of $K$). In the Higgs case with the same data $(G,V) = (K, \mathbb{A}^n)$, the generators were indexed by $\mathfrak{t}^*$ (characters of $T$). This is a manifestation of \emph{Gale duality} as a special case of 3D mirror symmetry.

Again by choosing a generic stability condition $\xi \in \mathfrak{t} $ one can take a resolution of the BFN Coulomb branch in the abelian case. The universal deformation $\mathcal{X}_C$ of $X_C$ is obtained by varying the $(T^\vee)$-moment map parameter, and its affinization is given by $\mathcal{Y}_C = T^*\mathbb{A}^n /\!\!/_0 T^\vee = \mathrm{Spec}\ \mathcal{O}(T^*\mathbb{A}^n)^{T^\vee} = \mathrm{Spec}\  k[x_i, y_i]^{T^\vee}$.

These fit into the following diagram (cf. \cite[3(ix)]{BFN2}): 
\begin{center}
    \begin{tikzcd}
        & X_C = \mu^{-1}(0)/\!\!/_\xi T^\vee \rar \dar & \mathcal{X}_C  = T^* \mathbb{A}^n /\!\!/_\xi T^\vee \dar \\ & Y_C = \mu^{-1}(0)/\!\!/_0 T^\vee \rar & \mathcal{Y}_C = T^*\mathbb{A}^n /\!\!/_0 T^\vee 
    \end{tikzcd}.
\end{center}
This diagram is compatible with the description of the deformation of Coulomb branches (see the discussion at the end of \cref{sssec:exmp-coulomb}) in terms of flavor symmetries: here we consider the extension of $K^{\vee}$ by $T^{\vee}$ into $(\mathbb{G}_m)^n$ acting on $\mathbb{A}^n$.

\begin{prop}[{\cite[3(ix), 4(ii, vii)]{BFN2}}]\label{prop:abelian-coulomb-is-hypertoric-quantum}
    The canonical quantization of the BFN Coulomb branch $\mathcal{A} = \mathcal{A}_\hbar((\mathbb{G}_m)^n, \mathbb{A}^n)^{T^\vee}$ is isomorphic as associative algebras to the hypertoric enveloping algebra $\bigoplus_{\nu \in \mathfrak{k}} \mathcal{D}(\mathbb{A}^n)_{\nu}$. 
\end{prop}

\subsubsection{The $D$-module of twisted traces}

The Frobenius-constant quantization structure on hypertoric varieties can be constructed in either the Higgs or the Coulomb description. In the Higgs case, the construction is simple, as the quantization is again given by (crystalline) differential operators, not unlike the case of  Springer resolutions.

Recall that the universal deformation $\mathcal{X}_H$ has coordinate ring given by $\mathcal{O}(T^*\mathbb{A}^n)^K = k[x_i, y_i]^K$, with the zero $T$-weight part given by $\mathcal{O}(\mathcal{X}_H)_0 = k[x_iy_i]$.
\begin{lem}
There is a Frobenius-constant quantization
\begin{equation}
    \Lambda_{Higgs} : \mathcal{O}(\mathcal{X}_H)^{(1)}_0 \to \mathcal{Z}(\mathcal{A}_0)
\end{equation}
of the abelian Higgs branch induced by the $p$-curvature endomorphism (the Artin--Schreier map) of the Weyl algebra, that is  $x_iy_i$ is sent to $x_i^p \partial_i^p = (x_i \partial_i)^p - \hbar^{p-1} (x_i\partial_i)$.
\end{lem}
\begin{proof}
In fact, one may define the Frobenius-constant quantization on the whole of $\mathcal{O}(\mathcal{X}_H) = \mathcal{O}(T^*\mathbb{A}^n)^K$ by restricting the canonical Frobenius-constant quantization $x_i \mapsto x_i^p$, $y_i \mapsto \partial_i^p$ of $\mathcal{O}(T^*\mathbb{A}^n)$. Restricted to the zeroth weight part, one can see that the action on the generators $x_iy_i$ is determined by the algebra map property of the canonical Frobenius-constant quantization, namely
\begin{equation}
    x_i y_i \mapsto x_i^p \partial_i^p = (x_i \partial_i)^p - \hbar^{p-1} (x_i \partial_i);
\end{equation}
for the last equality we use the combinatorial identity $x_i^k \partial_i^k = (x_i \partial_i) (x_i\partial_i + \hbar) \cdots (x_i\partial_i + (k-1)\hbar)$ and $\prod_{j \in \mathbb{F}_p} (x_i \partial_i + j \hbar) = (x_i \partial_i)^p - \hbar^{p-1} (x_i \partial_i)$. Therefore, we can simply define $\Lambda_{Higgs}$ by the restriction, which produces the Artin--Schreier map on the generators.
\end{proof}

The construction of the Frobenius-constant quantization structure for the hypertoric varieties in their description as Coulomb branches is given by a special case of Lonergan's theorem.
\begin{thm}[cf. {\cite[Example 3.16(3)]{lonergan}}]
    There is a Frobenius-constant quantization 
\begin{equation}
    \Lambda_{Coulomb} : \mathcal{O}(\mathcal{X}_C)^{(1)}_0 \to \mathcal{Z}(\mathcal{A}_0)
\end{equation}

    of the abelian BFN Coulomb branch over $k$ induced by the $p$-curvature endomorphism (the Artin--Schreier map) of the Weyl algebra,  that is  $x_iy_i$ is sent to $x_i^p \partial_i^p = (x_i \partial_i)^p - \hbar^{p-1} (x_i\partial_i)$.
\end{thm}
\begin{proof}
    We directly compute Lonergan's Frobenius-constant quantization, using the Steenrod linearity of the construction \cite[Proposition 3.30]{lonergan}. (Our $\Lambda_{Coulomb}$ is denoted $F_\hbar$ in \emph{loc. cit.}).

    We recall that the isomorphism between $\mathcal{A}_\hbar ((\mathbb{G}_m)^n , \mathbb{A}^n) \cong H_*^{(\mathbb{G}_m)^n\brO \rtimes \mathbb{G}_m} (\mathcal{R}_{(\mathbb{G}_m)^n, \mathbb{A}^n} ; k)$ (the quantized Coulomb branch) and the Weyl algebra $\mathcal{D}(\mathbb{A}^n) \cong k[\hbar] \langle x_i, \partial_i \rangle$ \cite[4(ii), Equation (4.8)]{BFN2} sends the equivariant parameters $\lambda_i$ for $1 \le i \le n$ to $x_i\partial_i$. In the classical limit, we have $\mathcal{M}_C((\mathbb{G}_m)^n ; \mathbb{A}^n) \cong T^*\mathbb{A}^{n}$. The (universal deformation of) abelian Coulomb branch $\mathcal{X}_C$ associated to the group $K$ in the exact sequence $K \to (\mathbb{G}_m)^n \to T$ is obtained as the GIT quotient of $T^*\mathbb{A}^n /\!\!/_0 T^\vee$; its $K^\vee$-weight zero part is a copy of $H^*_{\mathbb{G}_m^n} (\mathrm{pt};k) \cong k [\lambda_i]$ where the equivariant parameters $\lambda_i$ are sent to $x_iy_i \in k[x_iy_i] \cong \mathcal{D}(\mathbb{A}^n)_0$.
    
    Since Lonergan's quantization map $\Lambda_{Coulomb}$ (from \cref{thm:lonergan-quantization}) is unital (\cite[Remark 3.15 (ii)]{lonergan}) and Steenrod-linear (\cite[Proposition 3.30]{lonergan}), we have
    \begin{equation}
        \Lambda_{Coulomb}(a) = \mathrm{St}(a) \Lambda_{Coulomb}(1) = \mathrm{St}(a)
    \end{equation}
    for equivariant parameters $a \in H^*_{(\mathbb{G}_m)^n}(\mathrm{pt};k)$. It therefore suffices to compute the classical Steenrod operations on the generators $x_i y_i$. But for the equivariant parameters $\lambda_i \in H^2_{(\mathbb{G}_m)^n[\![z]\!] }(\mathcal{R};k)$, the general computation of Steenrod operations of degree $2$ classes (that admit integral lift) we have
    \begin{equation}
        \mathrm{St}(\lambda_i) = \lambda_i^p - \hbar^{p-1}\lambda_i \in H^{2p}_{(\mathbb{G}_m)^n[\![z]\!] \rtimes \mathbb{G}_m}(\mathcal{R};k),
    \end{equation}
    which is identified with $(x_i\partial_i)^p - \hbar^{p-1}(x_i\partial_i)$ in the hypertoric enveloping algebra. This is the desired result. \qedhere
    
\end{proof}

\begin{rem}[cf. {\cite[Example 3.16(3)]{lonergan}}]
    We can additionally compute the Steenrod operations of $x_i \in H_0^{G[\![z]\!]}(\mathcal{R};k)$, the monopole operators living in nonzero torus weight summands, as follows. By degree reasons, the total Steenrod power of $x_i$ is just $\mathrm{St}(x_i) = x_i^p$; indeed, this is the $\hbar^0$-term in the expansion of $\mathrm{St}(x_i)$, and the higher order terms in $\hbar$ involve coefficients in $H_*^{G[\![z]\!] \rtimes \mathbb{G}_m}(\mathcal{R};k)$ of negative degree, which are zero. Following the construction outlined in \cref{thm:lonergan-quantization}, this implies that $\Lambda_{Coulomb}(x_i) = x_i^p$.
\end{rem}

These computations verify \cref{conj:pcurv-for-trace} for hypertoric varieties.

\subsubsection{The quantum $D$-module}

In light of \cref{thm:qst-is-pcurv-onthenose}, to show that the quantum Steenrod operations agree with the $p$-curvature, it suffices to verify that the quantum multiplication operators by divisor classes have jointly simple spectrum. We use the explicit description of $H^2(X)$ for smooth hypertoric varieties $X$ obtained by \cite{harada-holm}.

\begin{prop}\label{prop:qst-is-pcurv-hypertoric}
        Let $X := T^*\mathbb{A}^n /\!\!/\!\!/\!\!/\!\!_{\chi, 0} K$ be a smooth hypertoric variety defined from the data $K \to (\mathbb{G}_m)^n \to T$ and $\chi \in \mathfrak{k}^*$, equipped with a Hamiltonian action of the torus $T$. Fix a degree $2$ class $x \in H^2_{T}(X;k)$; then for almost all $p$, we have
    \begin{equation}
        \Sigma_x^T = \nabla_x^p - t^{p-1} \nabla_x,
    \end{equation}
    that is, the quantum Steenrod operation on degree $2$ classes agree with the $p$-curvature.
\end{prop}

\begin{proof}
    It suffices to show that the collection of (equivariant) quantum multiplication operators $y \star$ have jointly simple spectrum for generic $p$. Since $y \star$ is a formal deformation (as commuting operators) of the classical (equivariant) cup product operator $y \cup$, it suffices to check that $y \cup$ have jointly simple spectrum.

    To fix the notation, consider the short exact sequence of algebraic groups
    $1 \to K \to (\mathbb{G}_m)^n \to T \to 1$
    defining the hypertoric variety. The image of the standard basis under the associated map of Lie algebras $\mathbb{Z}^n \to \mathfrak{t} = \mathrm{Lie}(T)$ are denoted by $\{ a_i \}_{1 \le i \le n} \in  \mathfrak{t}$. The choice of the stability condition $\chi \in \mathfrak{k}^*$ determines hyperplanes $H_i = \{ x \in \mathfrak{t}^*_{\mathbb{R}} : \langle x, a_i \rangle = - \chi_i \}$ where $\chi_i$ is the $i$th component of a choice of a lift of $\chi \in \mathfrak{k}^*$ to $\mathbb{Z}^n = \mathrm{Lie}((\mathbb{G}_m)^n)^*$; this hyperplane arrangement is well-defined up to simultaneous translation. The fixed point basis in the (localized) equivariant cohomology $H^*_{T \times \mathbb{G}_m}(X) \otimes \mathrm{Frac}H^*_{T \times \mathbb{G}_m} (\mathrm{pt})$ are indexed by the vertices of the hyperplane arrangement, and in particular the rank is finite \cite[Proposition 3.2]{harada-holm}. For a vertex $v$ in the hyperplane arrangement, we denote by $[v]$ the corresponding basis element (of cohomological degree $0$) in cohomology, corresponding to the fundamental class of the isolated torus fixed point associated to $v$.

    For each fixed point indexed by a vertex $v$, let $I_v \subseteq \{1, \dots, n\}$ be a subset of cardinality $|I_v| = \mathrm{rank}(T)$ such that $i \in I_v$ if and only if $v \in H_i$. For each pair $(v,i)$ such that $i \in I_v$, let $\eta_{v, i} \in \mathfrak{t}^* \cong H^2_T(\mathrm{pt})$ to be the unique element such that $\langle \eta_{v,i} , a_i \rangle = 1$ and $\langle \eta_{v,i} , a_j \rangle =0$ for $j \in I_v \setminus \{i\}$. 

    The main computation of \cite[Theorem 3.5]{harada-holm} shows that the divisor operators $y \cup$ for $y \in H^2_T(X)$ are generated by $n$ classes $\rho_i$, and in the fixed point basis their action is given by
    \begin{equation}
        \rho_i \cup [v] = \begin{cases}
           \eta_{v,i} + \left\langle \eta_{v,i}, \sum_{\langle v, a_j \rangle > - \chi_j} a_j \right\rangle \hbar & \langle v, a_i \rangle = \chi_i\\ 0 & \langle v, a_i \rangle > - \chi_i \\ \hbar & \langle v, a_i \rangle < - \chi_i
        \end{cases}.
    \end{equation}
    (Note that the first case $\langle v, a_i \rangle = \chi_i$ is exactly $v \in H_i$ or equivalently $i \in I_v$.) To see that these operators have jointly simple spectrum, we need to check that for a pair of distinct fixed points $v \neq v'$ there is $\rho_i$ such that the weights at $v$ and $v'$ are distinct. The computation immediately implies this: choose a hyperplane $H_i$ such that $v \in H_i$ but $v' \notin H_i$, and $\rho_i\cup$ acts by a nonzero $T$-weight on $v$ but by zero $T$-weight on $v'$.
\end{proof}

This result verifies \cref{conj:pcurv-for-quantum} for smooth hypertoric varieties.

\bibliographystyle{amsalpha}
\bibliography{ref}

\end{document}